\documentclass[12pt,letterpaper,reqno]{amsart}

\usepackage[T1]{fontenc}
\usepackage[utf8]{inputenc}

\usepackage{newtxtext}
\usepackage{newtxmath}

\usepackage[scaled=0.90]{helvet}

\usepackage{courier}

\usepackage{mathrsfs}

\usepackage[table]{xcolor}

\usepackage{microtype}                   
\usepackage[
    a4paper,
    top=2.5cm,
    bottom=2.5cm,
    left=2.5cm,
    right=2.5cm,
    headheight=14pt,
    headsep=1.2em,
    footskip=2.5em,
]{geometry}

\usepackage{fancyhdr}
\fancypagestyle{plain}{%
    \fancyhf{}
    
    \fancyfoot[C]{\small\thepage}
}

\usepackage{amsmath, amsthm}
\usepackage{mathtools}
\usepackage{bm}
\usepackage{nicefrac}
\usepackage{stmaryrd}

\theoremstyle{plain}
\newtheorem{theorem}{Theorem}
\newtheorem{lemma}[theorem]{Lemma}
\newtheorem{proposition}[theorem]{Proposition}
\newtheorem{corollary}[theorem]{Corollary}
\newtheorem{conjecture}{Conjecture}

\theoremstyle{definition}
\newtheorem{definition}[theorem]{Definition}

\newtheorem{problem}{Problem}

\theoremstyle{remark}

\newtheorem*{theorem*}{Theorem}
\newtheorem*{lemma*}{Lemma}
\newtheorem*{proposition*}{Proposition}
\newtheorem*{corollary*}{Corollary}
\newtheorem*{conjecture*}{Conjecture}
\newtheorem*{remark*}{Remark}

\AtBeginDocument{}

\usepackage{graphicx}
\usepackage{float}
\usepackage{booktabs}
\usepackage{diagbox}
\usepackage{caption}
\usepackage{tikz}
\usepackage{tikz-cd}       
\usetikzlibrary{arrows.meta, calc, positioning, decorations.pathmorphing,
                 shapes.geometric, patterns, shadows}

\usepackage[
    colorlinks  = true,
    linkcolor   = cyan!80!black,
    citecolor   = purple,
    urlcolor    = blue!60!black,
    backref     = page,
]{hyperref}
\usepackage[nameinlink, capitalize, noabbrev]{cleveref}

\makeatletter
\renewcommand{\@cite}[2]{%
    {\scriptsize\raise0.5pt\hbox{[{#1\if@tempswa , #2\fi}]}}%
}
\makeatother

\definecolor{backrefcolor}{HTML}{8B4513}
\definecolor{citpagecolor}{HTML}{D2691E}

\makeatletter
\renewcommand*{\backref}[1]{}

\renewcommand*{\backrefalt}[4]{%
    \ifcase #1 %
    \or {\hspace{0.8em}\footnotesize\color{black}(cit.\ on p.~#2).}%
    \else {\hspace{0.8em}\footnotesize\color{black}(cit.\ on pp.~#2).}%
    \fi
}

\makeatother

\definecolor{emailcolor}{HTML}{E6007E}
\makeatletter
\let\orig@email\email
\renewcommand{\email}[1]{\orig@email{\color{emailcolor}#1}}
\makeatother

\makeatletter
\renewcommand{\section}{\@startsection{section}{1}%
    {\z@}{-2.5ex \@plus -1ex \@minus -.2ex}{1.5ex \@plus .2ex}%
    {\normalfont\fontsize{16}{19}\bfseries}}
\renewcommand{\subsection}{\@startsection{subsection}{2}%
    {\z@}{-2ex \@plus -0.8ex \@minus -.2ex}{1ex \@plus .2ex}%
    {\normalfont\fontsize{13}{16}\bfseries}}
\renewcommand{\subsubsection}{\@startsection{subsubsection}{3}%
    {\z@}{-1.5ex \@plus -0.5ex \@minus -.1ex}{0.8ex \@plus .1ex}%
    {\normalfont\fontsize{11}{14}\bfseries\itshape}}

\renewcommand{\@settitle}{\begin{center}%
    \baselineskip20pt\relax
    \bfseries\fontsize{18}{22}\selectfont
    \@title
    \end{center}%
}
\let\uppercasenonmath\@gobble
\AtBeginDocument{}

\renewcommand{\@setauthors}{%
    \begingroup
    \trivlist
    \centering
    \fontsize{13}{16}\selectfont
    \@topsep40\p@\relax
    \advance\@topsep by -\baselineskip
    \item\relax
    \begin{tabular}[t]{@{}l@{\hspace{2em}}l@{}}
    Narges Ghareghani & Mehdi Golafshan\textsuperscript{$\ast$}\\[2pt]
    Morteza Mohammad-Noori\textsuperscript{$\dagger$} & Pouyeh Sharifani
    \end{tabular}
    \vspace{10pt}
    \endtrivlist
    \endgroup
}

\def\@setthanks{\def\thanks##1{\par##1}\thankses}

\makeatother

\definecolor{toccolor}{RGB}{27,58,92}

\makeatletter
\renewcommand{\tableofcontents}{%
    \par
    \begin{list}{}{%
        \leftmargin2.5pc \rightmargin2.5pc
        \listparindent\z@ \itemindent\z@
        \parsep\z@ \topsep6pt \partopsep\z@}
    \item\relax
    {\large\bfseries Contents}\par
    \vspace{4pt}%
    \@input{\jobname.toc}%
    \if@filesw
        \expandafter\newwrite\csname tf@toc\endcsname
        \immediate\openout \csname tf@toc\endcsname \jobname.toc\relax
    \fi
    \global\@nobreakfalse
    \end{list}%
}
\makeatother

\makeatletter
\renewcommand{\tocsection}[3]{%
    \indentlabel{\@ifnotempty{#2}{\color{toccolor}\ignorespaces#1 #2.\quad}}{\bfseries\color{toccolor}#3}}
\renewcommand{\tocsubsection}[3]{%
    \indentlabel{\@ifnotempty{#2}{\color{toccolor}\ignorespaces#1 #2.\quad}}{\color{toccolor}#3}}
\renewcommand{\tocappendix}[3]{%
    \indentlabel{\@ifnotempty{#2}{\color{toccolor}\ignorespaces#1 #2.\quad}}{\bfseries\color{toccolor}#3}}

\def\l@section{\@tocline{1}{2pt}{0pt}{}{}}
\def\l@subsection{\@tocline{2}{0pt}{1.5em}{}{}}

\def\@tocline#1#2#3#4#5#6#7{\relax
    \ifnum #1>\c@tocdepth
    \else
        \par \addpenalty\@secpenalty\addvspace{#2}%
        \begingroup \hyphenpenalty\@M \small
        \@ifempty{#4}{%
            \@tempdima\csname r@tocindent\number#1\endcsname\relax
        }{%
            \@tempdima#4\relax
        }%
        \parindent\z@ \leftskip#3\relax \advance\leftskip\@tempdima\relax
        \rightskip\@pnumwidth plus4em \parfillskip-\@pnumwidth
        #5\leavevmode\hskip-\@tempdima
        {\color{toccolor}#6}\nobreak\relax
        \leaders\hbox{$\color{toccolor}\m@th
            \mkern 4.5mu\hbox{\normalfont\small\color{toccolor}.}%
            \mkern 4.5mu$}\hfill
        \hbox to\@pnumwidth{\@tocpagenum{\color{toccolor}#7}}\par
        \nobreak
        \endgroup
    \fi}
\makeatother

\usepackage{enumitem}
\setlist{leftmargin=*, itemsep=3pt, parsep=1pt}
\setlist[enumerate,1]{label=(\roman*)}
\setlist[enumerate,2]{label=(\alph*)}

\newcommand{\N}{\mathbb{N}}
\newcommand{\Z}{\mathbb{Z}}

\DeclareFontFamily{U}{eus}{}
\DeclareFontShape{U}{eus}{m}{n}{<->eusm10}{}
\newcommand{\A}{\text{\usefont{U}{eus}{m}{n}A}}

\DeclarePairedDelimiter{\abs}{\lvert}{\rvert}

\DeclarePairedDelimiter{\floor}{\lfloor}{\rfloor}

\DeclarePairedDelimiter{\set}{\{}{\}}

\newcommand{\defeq}{\coloneqq}

\definecolor{tabrefcolor}{named}{blue}
\definecolor{secrefcolor}{named}{red}
\newcommand{\tabref}[1]{{\small\hyperref[#1]{\textcolor{tabrefcolor}{Table~\ref*{#1}}}}}
\newcommand{\secref}[1]{{\small\hyperref[#1]{\textcolor{secrefcolor}{Section~\ref*{#1}}}}}
\newcommand{\subsecref}[1]{{\small\hyperref[#1]{\textcolor{secrefcolor}{Subsection~\ref*{#1}}}}}
\newcommand{\thmref}[2]{{\small\hyperref[#1]{#2~\ref*{#1}}}}

\definecolor{abstractbg}{gray}{0.93}

\makeatletter
\renewenvironment{abstract}{%
    \ifx\maketitle\relax
        \ClassWarning{\@classname}{Abstract should precede
            \protect\maketitle}%
    \fi
    \global\setbox\abstractbox=\vtop\bgroup
        \normalfont\small
        \list{}{\labelwidth\z@
            \leftmargin5pc \rightmargin\leftmargin
            \listparindent\normalparindent \itemindent\z@
            \parsep\z@ \@plus\p@
            
        }%
        \item\relax
        \begin{center}\textbf{Abstract}\end{center}%
        \vspace{-2pt}%
        \noindent\ignorespaces
}{%
    \endlist\egroup
    \ifx\@setabstract\relax \@setabstracta \fi
}

\let\orig@setabstracta\@setabstracta
\renewcommand{\@setabstracta}{%
    \orig@setabstracta
}
\makeatother

\title{Enumeration of Factor Occurrences in $k$-Bonacci Words over an Infinite Alphabet}

\author{Narges Ghareghani}
\address{School of Mathematics, Statistics, and Computer Science,
         University of Tehran, Tehran, Iran}
\email{ghareghani@ut.ac.ir}

\author{Mehdi Golafshan}
\thanks{${}^{\ast}$Supported by the FNRS Research
grant T.196.23 (PDR)}
\address{Department of Mathematics, University of Li\`ege,
         Li\`ege, Belgium}
\email{mgolafshan@uliege.be}

\author{Morteza Mohammad-Noori}
\thanks{${}^{\dagger}$Corresponding author}
\address{School of Mathematics, Statistics, and Computer Science,
         University of Tehran, Tehran, Iran}
\email{mmnoori@ut.ac.ir}

\author{Pouyeh Sharifani}
\address{School of Mathematics, Statistics, and Computer Science,
         University of Tehran, Tehran, Iran}
\email{pouyeh.sharifani@gmail.com}

\date{}

\begin{document}

\begin{abstract}
We study the $k$-Bonacci word over the infinite alphabet~$\N$.
Since the alphabet is infinite, the usual factor complexity is
infinite and does not provide any information. We therefore
investigate factor occurrence statistics in the finite iterates.
For $k \ge 3$, we obtain closed forms for the generating
functions (with respect to the iteration index) that count the
number of occurrences of an arbitrary digit in the $n$th iterate.
We then characterize the complete set of length-$2$ factors
occurring in the infinite word and compute, for each such factor,
a closed form for the generating function encoding its number of
occurrences across all finite iterates. As a consequence, the
associated counting sequences satisfy uniform $(k\!-\!1)$-step
Fibonacci-type recurrences and admit a description in terms of
$(k\!-\!1)$-Bonacci enumeration phenomena, including
self-convolution structures.
\end{abstract}

\maketitle

\vspace{-4pt}
\begin{list}{}{%
    \leftmargin3.5pc \rightmargin3.5pc
    \listparindent0pt \itemindent0pt
    \parsep4pt \topsep0pt \partopsep0pt}
\item\relax

\item\relax
\end{list}
\vspace{4pt}

\tableofcontents

\newpage
\section{Introduction}\label{Int}

We write $\N = \set{0, 1, 2, \dots}$.
Let $\A = \set{0, 1}$ and let
$\varphi \colon \A^{\ast} \to \A^{\ast}$ be the
morphism defined by
\[
  \varphi(0) = 0\,1, \qquad \varphi(1) = 0.
\]
The (classical) Fibonacci word
$\mathbf{F}$
({\small\color{teal}\underline{\href{https://oeis.org/A003849}{\textcolor{teal}{A003849}}}}) is the unique right-infinite
fixed point of~$\varphi$; equivalently,
\[
  \mathbf{F} = \varphi^{\omega}(0)
             = \lim_{n \to \infty} \varphi^{n}(0)
             = 0100101001001010010\cdots.
\]

The \emph{factor complexity} of an infinite word~$\mathbf{u}$ is
the function $\mathrm{p}_{\mathbf{u}}(n)$ counting the number of
distinct factors of length~$n$. The Fibonacci word~$\mathbf{F}$ is
the prototypical Sturmian word, satisfying
$\mathrm{p}_{\mathbf{F}}(n) = n + 1$ for all
$n \in \N$~\cite{A} --- the minimum complexity for any aperiodic
word.

Classical monographs on combinatorics on words and symbolic
dynamics~\cite{B,C}, together with recent surveys at the interface
of combinatorics, automata, and number theory~\cite{D}, provide
comprehensive accounts of the Fibonacci word~$\mathbf{F}$,
including its realization as the coding of an irrational rotation
and as the symbolic coding of suitable billiard trajectories in the
unit square.

For general $k \ge 2$, Rauzy~\cite{E} (case $k = 3$; see
also~\cite{F}) and Sirvent~\cite{G} introduced the
\emph{$k$-Bonacci morphism}
$\varphi_k \colon \A_k^{\ast} \to \A_k^{\ast}$ by
$\varphi_k(i) = 0\,(i{+}1)$ for $0 \le i \le k-2$ and
$\varphi_k(k-1) = 0$. Setting
$\mathbf{F}_n^{(k)} \defeq \varphi_k^{n}(0)$, one obtains the
$k$-fold concatenation rule
\[
  \mathbf{F}_n^{(k)}
  = \mathbf{F}_{n-1}^{(k)}\,\mathbf{F}_{n-2}^{(k)}
    \cdots \mathbf{F}_{n-k}^{(k)},
  \qquad \forall\, n \ge k,
\]
so the lengths
$\abs{\mathbf{F}_n^{(k)}}$ satisfy the classical $k$-Bonacci
recurrence. The \emph{infinite $k$-Bonacci word} is
\[
  \mathbf{F}^{(k)} = \varphi_k^{\omega}(0)
                    = \lim_{n \to \infty} \varphi_k^{n}(0).
\]

The $k$-Bonacci word $\mathbf{F}^{(k)}$ is a strict episturmian
(equivalently, Arnoux--Rauzy) word over~$k$
letters~\cite{H,I}: it is aperiodic, its factor set is closed
under reversal, and
$\mathrm{p}_{\mathbf{F}^{(k)}}(n) = (k-1)n + 1$ for all
$n \in \N$. For $k = 2$ this recovers the Sturmian case. For
further background on episturmian and Arnoux--Rauzy words we
refer the reader to the survey~\cite{J}.

More recently, several authors have considered $k$-Bonacci words
over the infinite alphabet~$\N$, obtained by a ``lift'' of the
finite-alphabet substitution in which new symbols are created at
each morphic iteration. Since the alphabet is~$\N$, each letter
is simultaneously a natural number; to emphasize their role as
symbols we refer to them as \emph{digits} (see
\subsecref{Pre}). Expressions such as~$2i$ or~$ki+j$ thus denote
individual digits, while~$i$, $j$, and~$k$ remain integer
parameters.

For $k = 2$, Zhang \emph{et al.}~\cite{L}
introduced \emph{the Fibonacci word over an infinite alphabet},
denoted~$\mathbf{W}^{(2)}$
({\small\color{teal}\underline{\href{https://oeis.org/A104324}{\textcolor{teal}{A104324}}}}), as the right-infinite fixed point of
the morphism
$\phi_2 \colon \N^{\ast} \to \N^{\ast}$ defined by
\[
  \phi_2(2i) = (2i)(2i+1),
  \qquad
  \phi_2(2i+1) = 2i+2,
  \qquad \forall\, i \in \N,
\]
where juxtaposition denotes concatenation. For example, starting
from~$0$, one obtains
\[
  \mathbf{W}^{(2)} = \phi_2^{\omega}(0)
  = 0\,1\,2\,2\,3\,2\,3\,4\,2\,3\,4\,4\,5\,2\,3\,4\,4\,5\,4\,5\,6
    \cdots,
\]
so that each even letter~$2i$ is preserved and produces a new odd
letter~$2i+1$, while each odd letter~$2i+1$ is replaced by the new
even letter~$2i+2$.

More generally, for any integer $k \ge 3$, Ghareghani
\emph{et al.}~\cite{M} defined \emph{the infinite alphabet
$k$-Bonacci word}~$\mathbf{W}^{(k)}$ as the fixed point
$\mathbf{W}^{(k)} = \phi_k^{\omega}(0)$ of the morphism
$\phi_k \colon \N^{\ast} \to \N^{\ast}$ given, for $i \in \N$ and
$0 \le j \le k-1$, by
\[
  \phi_k(ki+j) =
  \begin{cases}
    (ki)(ki+j+1), & \text{if } 0 \le j \le k-2;\\[2pt]
    ki+k,         & \text{if } j = k-1.
  \end{cases}
\]
By construction, every integer occurs as a letter
in~$\mathbf{W}^{(k)}$.

The reduction map $\pi_k \colon \N \to \A_k$,
$n \mapsto n \bmod k$, extends letterwise to words and satisfies
$\pi_k(\mathbf{W}^{(k)}) = \mathbf{F}^{(k)}$. In particular,
$\pi_2(\mathbf{W}^{(2)}) = \mathbf{F}$. For example,
$\mathbf{W}^{(3)}$ begins with $0102013\cdots$, which projects
under~$\pi_3$ to $0102010\cdots$.

Equivalently, writing
$\mathbf{W}^{(k)} = (w_n)_{n \in \N}$, one can view it as a
\emph{data word} $(\ell_n, d_n)_{n \in \N}$ where the finite label
is $\ell_n = \pi_k(w_n) \in \A_k$ and the data value is
$d_n = w_n \in \N$; in this setting, one typically compares data
values using predicates such as equality (and, in ordered variants,
the natural order), as in the standard data-word frameworks
of~\cite{O,N}.

Many combinatorial properties of~$\mathbf{W}^{(k)}$ project
under~$\pi_k$ to those of~$\mathbf{F}^{(k)}$~\cite{M,P,Q}.
However, the ordinary factor complexity of~$\mathbf{W}^{(k)}$ is
already infinite ---
$\mathrm{p}_{\mathbf{W}^{(k)}}(n) = \infty$ for all
$n \in \N_{>0}$ --- since $\mathbf{W}^{(k)}$ contains infinitely
many distinct letters. For instance, the factors $22$, $44$,
$66$, \dots\ in~$\mathbf{W}^{(2)}$ all project to the single
factor~$00$ under~$\pi_2$~\cite{U}.

Classical recurrence properties may also fail over an infinite
alphabet. For instance, $\mathbf{W}^{(2)}$ is not uniformly
recurrent in the usual sense: the letter~$1$ occurs only once (as
the second symbol), so the one-letter factor~$1$ appears only
finitely many times, even though its residue class modulo~$2$
continues to recur through larger odd letters.

When the alphabet is countably infinite, counting words by
cardinality is not informative: for every fixed length $m \ge 1$
there are already countably many words of length~$m$. Accordingly,
for infinite-alphabet morphic words such as~$\mathbf{W}^{(k)}$,
the combinatorial emphasis shifts from enumerating distinct factors
to analyzing \emph{pattern statistics}, such as
factor-occurrence counts across the finite iterates. In particular,
since $\mathbf{W}^{(k)}$ contains infinitely many distinct letters,
its ordinary factor complexity is infinite, so occurrence statistics
provide a more sensitive measure of structure.

\subsection{Our contributions}\label{subsec:contrib}

Berstel's survey on Fibonacci words~\cite{B} distinguishes two
broad research directions: (i)~the study of factors and
combinatorial structure (special factors, repetitions, complexity,
\emph{etc.}), and (ii)~numeration and arithmetic aspects
(normalization, addition/subtraction, automata/transducers). In the
infinite-alphabet setting, where standard factor complexity is
infinite, the first direction is naturally approached through
\emph{factor-occurrence statistics}. In this paper we follow this
viewpoint for the infinite word~$\mathbf{W}^{(k)}$. Our main
contributions are as follows.
\begin{itemize}
  \item \textbf{Digit-occurrence enumeration.}
    For each $m \in \N$ and each digit $d \in \N$, let
    $\abs{W_m^{(k)}}_d$ denote the number of occurrences
    of~$d$ in the finite word~$W_m^{(k)}$. We derive a unified
    recurrence for these quantities and solve it systematically
    using generating functions. This yields explicit closed
    forms showing that the dependence on~$d$ is governed by its
    quotient and remainder upon division by~$k$ (see
    \secref{section:4}, in particular
    {\small\hyperref[thm:Cdigit]{Theorem~\ref*{thm:Cdigit}}}).

  \item \textbf{Characterization of length-$2$ factors.}
    We determine the complete set of length-$2$ factors occurring
    in~$\mathbf{W}^{(k)}$ and organize them into three natural
    families corresponding to factors occurring inside blocks and
    those created at block boundaries in the recursive
    decomposition of the iterates (see \subsecref{subsection:5.1}
    and {\small\hyperref[thm:Fac2-characterization]{Theorem~\ref*{thm:Fac2-characterization}}}).

  \item \textbf{Enumeration of length-$2$ factor occurrences.}
    For each length-$2$ factor occurring
    in~$\mathbf{W}^{(k)}$, we compute a closed form for the
    generating function for its occurrence counts
    in the finite iterates~$W_m^{(k)}$, $m \in \N$. As a
    consequence, the associated counting sequences satisfy uniform
    $(k\!-\!1)$-step Fibonacci-type recurrences and admit a
    description in terms of $(k\!-\!1)$-Bonacci phenomena,
    including self-convolutions (i.e., Cauchy products of a
    counting sequence with itself); see \subsecref{subsection:5.2}.
\end{itemize}

\smallskip
\noindent\textbf{Methods.}\enspace
The proofs rely on the recursive block decomposition of the
iterates developed in \secref{sec:prelim}. We separate occurrences
fully contained in a single block from those interacting with block
boundaries, which leads to tractable linear recurrences for
occurrence counts. These recurrences are then solved uniformly by
passing to generating functions. A key technical device is
the shift operator (addition of~$k$ to every letter), which
transports factors between iterates and reduces the enumeration of
larger digits and factors to smaller ones.

\subsection{Related work}\label{subsec:related}

Although $\mathbf{W}^{(k)}$ is defined over an infinite alphabet,
many classical factor-theoretic properties remain accessible,
often by exploiting its morphic self-similarity. We briefly
highlight several results on the description and enumeration of
occurrences of certain families of factors.
\begin{itemize}
  \item \textbf{Square factors.}
    Glen \emph{et al.}~\cite{R} proved that
    $\mathbf{W}^{(2)}$ is cube-free with critical exponent~$2$,
    and that the total number of square occurrences (counted with
    multiplicity) in the $n$th iterate satisfies
    $T(W_n^{(2)}) = f_n - 1$, where $(f_n)_{n \in \N}$ is the
    Fibonacci sequence. Ghareghani and Sharifani~\cite{P}
    extended this to general~$k$, showing that
    $\mathbf{W}^{(k)}$ is cube-free for every $k > 2$ with
    critical exponent
    $E(\mathbf{W}^{(k)}) = 3 - 3/(2^{k}-1)$, and determined
    all factors achieving this exponent.

  \item \textbf{Palindromic factors.}
    Zhang \emph{et al.}~\cite{L} showed that $\mathbf{W}^{(2)}$
    has no palindrome of length greater than~$3$; Ghareghani
    \emph{et al.}~\cite{M} extended this to all~$k$, proving that
    $\mathbf{W}^{(k)}$ admits palindromes only of finitely many
    lengths. This contrasts sharply with the episturmian
    word~$\mathbf{F}^{(k)}$, which has palindromes of unbounded
    length.

  \item \textbf{Kernel words and special factors.}
    Kernel words and gap sequences, studied for the finite-alphabet
    Tribonacci and $k$-Bonacci words in~\cite{T,S}, have been
    extended to the infinite alphabet by Zhang~\cite{Q}, who
    describes the kernel words of~$\mathbf{W}^{(3)}$ and shows
    that the associated spacing phenomena persist.

  \item \textbf{Lyndon factors.}
    Glen \emph{et al.}~\cite{R} showed that each iterate
    $W_n^{(2)}$ is itself a Lyndon word and gave explicit formulas
    for the number of Lyndon factors by initial digit; for
    instance, $L_0(W_n^{(2)}) = f_{n+2}$. Analogous results for
    general~$k$ remain open.
\end{itemize}

\subsection{Notations and conventions}\label{Pre}

We write $\N_{>0} = \set{1, 2, \dots}$. Throughout this paper, we
work over the alphabet~$\N$ and refer to its elements as
\emph{digits}. We denote by~$\N^{\ast}$ the set of finite words
over~$\N$, with~$\varepsilon$ the empty word, and
by~$\N^{\N}$ the set of right-infinite words. The length of
$W \in \N^{\ast}$ is~$\abs{W}$; concatenation is written by
juxtaposition. Finite words are denoted by capital
letters $U, V, W, \dots$, and right-infinite words by boldface
$\mathbf{U}, \mathbf{V}, \mathbf{W}, \dots$. When writing
explicit examples, digits may exceed~$9$; to avoid ambiguity,
we occasionally insert separators when convenient.

A finite word $U \in \N^{\ast}$ is a \emph{factor} of
$V \in \N^{\ast} \cup \N^{\N}$ if $V = XUY$ for some
$X \in \N^{\ast}$ and
$Y \in \N^{\ast} \cup \N^{\N}$; we write
$U \sqsubseteq V$. We write $U \le_{\mathrm{pre}} V$
(resp.\ $U \le_{\mathrm{suf}} V$) if~$U$ is a prefix (resp.\
suffix) of~$V$. The set of factors of~$V$ of length~$m$ is
$\mathrm{Fac}_m(V) = \set*{U \in \N^{\ast} : \abs{U} = m
  \text{ and } U \sqsubseteq V}$.

For a non-empty word
$B \in \N^{\ast} \setminus \set{\varepsilon}$, the
\emph{occurrence count}~$\abs{W}_B$ is the number of (possibly
overlapping) occurrences of~$B$ in~$W$:
\[
  \abs{W}_B
  \defeq \#\bigl\{\, X \in \N^{\ast} :
    W = XBY \text{ for some } Y
    \in \N^{\ast} \,\bigr\}.
\]

The \emph{shift operator} adds a constant to every digit. For
$n \in \N$ and
$W = w_0 w_1 \cdots w_{m-1} \in \N^{\ast}$, define
\[
  n \oplus W
  \defeq (w_0 + n)\,(w_1 + n) \cdots (w_{m-1} + n);
\]
the definition extends to right-infinite words and to finite sets
$S \subseteq \N^{\ast} \setminus \set{\varepsilon}$ by setting
$n \oplus S \defeq \set*{n \oplus U : U \in S}$.
For example, $2 \oplus 0102013 = 2324235$ and
$3 \oplus \set{01, 22, 0102} = \set{34, 55, 3435}$.

We adopt the standard Iverson bracket notation: for any
predicate~$P$, the symbol~$[P]$ equals~$1$ if~$P$ is true
and~$0$ otherwise. In particular,
$[\,n = m\,] = \delta_{n,m}$ is the Kronecker delta.

\section{Preliminaries}\label{sec:prelim}

For $k \ge 2$, set
$W_n^{(k)} \defeq \phi_k^{\,n}(0)$ for $n \in \N$ and
$\mathbf{W}^{(k)} \defeq \phi_k^{\,\omega}(0)$. Each
$W_n^{(k)}$ is a prefix of~$\mathbf{W}^{(k)}$ of
length~$f_{n+k}^{(k)}$, where
$\bigl(f_m^{(k)}\bigr)_{m \in \N}$ denotes the $k$-Bonacci (or
$k$-generalized Fibonacci) numbers~\cite[Eq.~(1)]{1} defined by
$f_0^{(k)} = \cdots = f_{k-2}^{(k)} = 0$,
$f_{k-1}^{(k)} = 1$, and
\[
  f_m^{(k)}
  = f_{m-1}^{(k)} + \cdots + f_{m-k}^{(k)},
  \qquad \forall\, m \ge k.
\]
For example, when $k = 3$,
\[
\begin{aligned}
  W_0^{(3)} &= 0,\\
  W_1^{(3)} &= 01,\\
  W_2^{(3)} &= 0102,\\
  W_3^{(3)} &= 0102013,\\
  W_4^{(3)} &= 0102013010234,\ \text{etc.}
\end{aligned}
\]
The following lemma, obtained by combining Lemmas~4 and~6
in~\cite{M}, provides a recursive decomposition of the finite
$k$-Bonacci words, mirroring the usual numerical $k$-Bonacci
recurrence.

\begin{lemma}[{\cite[Lemma~4 and~6]{M}}]\label{lem:W-decomp}
Let $k \ge 2$ and let $n \in \N_{>0}$. Then
\[
  W_n^{(k)} =
  \begin{cases}
    W_{n-1}^{(k)}\, W_{n-2}^{(k)} \cdots W_0^{(k)}\, n,
      & \text{if } n < k;\\[1ex]
    W_{n-1}^{(k)}\, W_{n-2}^{(k)} \cdots W_{n-k+1}^{(k)}\,
      \bigl(k \oplus W_{n-k}^{(k)}\bigr),
      & \text{if } n \ge k.
  \end{cases}
\]
\end{lemma}

To illustrate the lemma for $k = 3$, consider the iterates listed
above. When $n = 1 < k$, the first branch yields
$W_1^{(3)} = W_0^{(3)}\, 1 = 01$. When $n = 4 \ge k$, the second
branch gives
\[
  W_4^{(3)}
  = W_3^{(3)}\, W_2^{(3)}\,
    \bigl(3 \oplus W_1^{(3)}\bigr)
  = 0102013 \cdot 0102 \cdot 34
  = 0102013010234.
\]

\begin{definition}[{\cite[Definition~10]{M}}]\label{def:types}
Let $k \ge 2$ and $n \in \N_{>0}$, and set
$s \defeq \max\set{n - k + 1,\, 0}$.
Each occurrence of a factor $U \sqsubseteq W_n^{(k)}$ is
classified according to its position relative to the block
decomposition of
{\small\hyperref[lem:W-decomp]{Lemma~\ref*{lem:W-decomp}}}.
A given factor may admit occurrences of different types.
We say that an occurrence of~$U$ is
\begin{enumerate}[label=\textup{(\Roman*)}]
  \item\label{item:included} \emph{included} if $U = n$, or if
    $U \sqsubseteq W_i^{(k)}$ for some $s \le i \le n-1$, or,
    when $n \ge k$, if
    $U \sqsubseteq k \oplus W_{n-k}^{(k)}$;

  \item\label{item:bordering} \emph{bordering of type~$j$} if there exists an
    integer~$j$ with $s + 1 \le j \le n - 1$ and a factorization
    $U = X_j Y_j$ with
    $X_j, Y_j \in \N^{\ast} \setminus \set{\varepsilon}$ such that
    \[
      X_j \le_{\mathrm{suf}} W_j^{(k)}
      \quad\text{and}\quad
      Y_j \le_{\mathrm{pre}}
        \bigl(W_{j-1}^{(k)}\, W_{j-2}^{(k)} \cdots
              W_s^{(k)}\bigr);
    \]

  \item\label{item:straddling} \emph{straddling} if it admits a factorization
    $U = PQ$ with
    $P, Q \in \N^{\ast} \setminus \set{\varepsilon}$ such that
    \[
      P \le_{\mathrm{suf}}
        \bigl(W_{n-1}^{(k)}\, W_{n-2}^{(k)} \cdots
              W_s^{(k)}\bigr)
      \quad\text{and}\quad
      \begin{cases}
        Q = n, & \text{if } n < k;\\
        Q \le_{\mathrm{pre}}
          \bigl(k \oplus W_{n-k}^{(k)}\bigr),
          & \text{if } n \ge k.
      \end{cases}
    \]
\end{enumerate}
\end{definition}

To illustrate the three types, take $k = 3$ and $n = 4$, giving
$s = 2$ and the decomposition
\[
  W_4^{(3)}
  = \underbrace{W_3^{(3)}}_{0102013}
    \underbrace{W_2^{(3)}}_{0102}
    \underbrace{3 \oplus W_1^{(3)}}_{34}\,.
\]
The length-$2$ factor~$02$ lies entirely inside~$W_2^{(3)} = 0102$
and thus has an included occurrence
in the sense of~\ref{item:included}.
The factor~$30$ lies across the junction
$W_3^{(3)} \mid W_2^{(3)}$: its decomposition $30 = 3 \cdot 0$
satisfies $3 \le_{\mathrm{suf}} W_3^{(3)}$ and
$0 \le_{\mathrm{pre}} W_2^{(3)}$, making it a bordering
occurrence of type~$3$ as in~\ref{item:bordering}.
Finally, the factor~$23$ crosses the last boundary: writing
$23 = 2 \cdot 3$ we have
$2 \le_{\mathrm{suf}} (W_3^{(3)}\, W_2^{(3)})$ and
$3 \le_{\mathrm{pre}} (3 \oplus W_1^{(3)}) = 34$, so it is a
straddling occurrence in the sense of~\ref{item:straddling}.

\begin{remark*}\label{rem}
\begin{enumerate}[label=\textup{(\roman*)}]
  \item\label{item:rem-bordering} If an occurrence of~$U$ is bordering of type~$j$ (in the sense
    of~\ref{item:bordering} above), say $U = X_j Y_j$ with
    $X_j, Y_j \neq \varepsilon$, then the length-$2$ factor
    $j\,0$ occurs in~$U$.

  \item When it is convenient to specify the decomposition
    $U = PQ$ in~\ref{item:straddling}, we refer to the occurrence
    as $(P, Q)$-straddling.
\end{enumerate}
\end{remark*}

The following lemma, which will be used repeatedly, shows that the
largest digit in each iterate is confined to a single position.

\begin{lemma}[{\cite[Lemma~9]{M}}]\label{lem:largest-letter}
Let $k \ge 2$ and $n \in \N_{>0}$. The largest digit
in~$W_n^{(k)}$ is~$n$, and it appears exactly once, as the final
letter. Consequently, every factor~$U \sqsubseteq W_n^{(k)}$
containing~$n$ satisfies
\[
  U \le_{\mathrm{suf}} W_n^{(k)}
  \quad\text{and}\quad
  n \le_{\mathrm{suf}} U.
\]
\end{lemma}

We introduce two auxiliary sequences. For each $k \ge 2$, let
$\bigl(h_n^{(k)}\bigr)_{n \in \N}$ be defined by
$h_0^{(k)} = 1$,\;
$h_n^{(k)} = 2^{\,n-1}$ for $1 \le n \le k-1$, and
\[
  h_n^{(k)}
  = h_{n-1}^{(k)} + \cdots + h_{n-k}^{(k)},
  \qquad \forall\, n \ge k;
\]
and let $\bigl(g_n^{(k)}\bigr)_{n \in \N}$ be defined by
$g_0^{(k)} = 0$,\;
$g_n^{(k)} = 2^{\,n-1}$ for $1 \le n \le k-1$, and
\[
  g_n^{(k)}
  = g_{n-1}^{(k)} + \cdots + g_{n-k}^{(k)},
  \qquad \forall\, n \ge k.
\]
Let
\[
  H_k(y) \defeq \sum_{n \in \N} h_n^{(k)}\, y^n
  \qquad\text{and}\qquad
  G_k(y) \defeq \sum_{n \in \N} g_n^{(k)}\, y^n
\]
denote the generating functions of
$\bigl(h_n^{(k)}\bigr)_{n \in \N}$ and
$\bigl(g_n^{(k)}\bigr)_{n \in \N}$, respectively. A standard
generating function extraction yields
\begin{align}
  H_k(y) &= \frac{1}{1 - y - \cdots - y^k},
    \label{eq:Hk}\\
  G_k(y) &= (1 - y^k)\, H_k(y) - 1.
    \label{eq:GkHk}
\end{align}

\section{Factor Occurrence Counts and the Shift Identity}%
\label{section:3}

Fix $k \ge 2$ and a non-empty word
$B \in \N^{\ast} \setminus \set{\varepsilon}$. For $n \in \N$,
let $c^{(k)}(B;\, n) \defeq \abs{W_n^{(k)}}_B$,
and define the generating function
$C_B^{(k)}(y) \defeq \sum_{n \in \N} c^{(k)}(B;\, n)\, y^n$.
For example, when $k = 3$,
\[
  C_0^{(3)}(y)
  = 1 + y + 2y^2 + 3y^3 + 5y^4 + \cdots
  \qquad\text{and}\qquad
  C_{01}^{(3)}(y)
  = y + y^2 + 2y^3 + 3y^4 + \cdots.
\]

\begin{theorem}\label{thm:CBshift}
Let $k \ge 3$ and let
$B \in \N^{\ast} \setminus \set{\varepsilon}$ be a factor
of~$\mathbf{W}^{(k)}$.
\begin{enumerate}[label=\textup{(\alph*)}]
  \item\label{item:shift-factor}
    The shifted word $k \oplus B$ is again a factor
    of~$\mathbf{W}^{(k)}$.
  \item\label{item:shift-ogf}
    If, in addition, for every $n \ge k$ every occurrence of
    $k \oplus B$ in~$W_n^{(k)}$ is included (equivalently,
    $k \oplus B$ has no bordering or straddling occurrence
    in~$W_n^{(k)}$), then
    \begin{equation}\label{eq:CBshift}
      C_{k \oplus B}^{(k)}(y)
      = y^k\, H_{k-1}(y)\, C_B^{(k)}(y).
    \end{equation}
\end{enumerate}
\end{theorem}

\begin{proof}
Choose $N \in \N$ such that $B \sqsubseteq W_N^{(k)}$.
By {\small\hyperref[lem:W-decomp]{Lemma~\ref*{lem:W-decomp}}} applied with $n = N + k$, the
word~$W_{N+k}^{(k)}$ has suffix $k \oplus W_N^{(k)}$; hence
\[
  k \oplus B
  \sqsubseteq k \oplus W_N^{(k)}
  \sqsubseteq W_{N+k}^{(k)}
  \sqsubseteq \mathbf{W}^{(k)},
\]
proving~\ref{item:shift-factor}.

For~\ref{item:shift-ogf}, fix $n \ge k$ and use
{\small\hyperref[lem:W-decomp]{Lemma~\ref*{lem:W-decomp}}}:
\[
  W_n^{(k)}
  = W_{n-1}^{(k)}\, W_{n-2}^{(k)} \cdots W_{n-k+1}^{(k)}\,
    \bigl(k \oplus W_{n-k}^{(k)}\bigr).
\]
Since $k \oplus B$ is neither bordering nor straddling, every
occurrence of $k \oplus B$ in~$W_n^{(k)}$ is contained in exactly
one of the~$k$ factors on the right-hand side. Since $\oplus$
preserves occurrence counts,
$\abs{k \oplus W_{n-k}^{(k)}}_{k \oplus B}
= c^{(k)}(B;\, n-k)$, so
\begin{equation}\label{eq:rec-cBshift}
  c^{(k)}(k \oplus B;\, n)
  = \sum_{j=1}^{k-1} c^{(k)}(k \oplus B;\, n-j)
    \;+\; c^{(k)}(B;\, n-k).
\end{equation}
Since $k \oplus B$ contains a digit at least~$k$, it cannot occur
in~$W_n^{(k)}$ for $0 \le n \le k-1$; hence
$c^{(k)}(k \oplus B;\, n) = 0$ for $0 \le n \le k-1$.
Multiplying~\eqref{eq:rec-cBshift} by~$y^n$, summing over
$n \ge k$, and recognizing $H_{k-1}(y)$ via~\eqref{eq:Hk}
gives~\eqref{eq:CBshift}.
\end{proof}

\begin{corollary}\label{cor:CBshift-digit}
Let $k \ge 3$ and $b \in \N$. Then
\[
  C_{k+b}^{(k)}(y) = y^k\, H_{k-1}(y)\, C_b^{(k)}(y).
\]
\end{corollary}

\begin{proof}
Since $\abs{b} = \abs{k \oplus b} = 1$, bordering and straddling
occurrences are impossible, so every occurrence of $k \oplus b$ is
included and
{\small\hyperref[thm:CBshift]{Theorem~\ref*{thm:CBshift}}}\,\ref{item:shift-ogf}
applies with $B = b$.
\end{proof}

\section{Digit Occurrence Generating Functions}%
\label{section:4}

We specialize to the case $\abs{B} = 1$, writing
$c^{(k)}(d;\, n)$ for the number of occurrences of digit~$d$
in~$W_n^{(k)}$.

\begin{lemma}\label{lem:cpn-unified}
Let $k \ge 2$ and $d \in \N$. For every $n \in \N$,
\begin{equation}\label{eq:cpn-unified}
  c^{(k)}(d;\, n)
  = \sum_{i=\max\{n-k+1,\,0\}}^{n-1} c^{(k)}(d;\, i)
    \;+\; [d \ge k]\; c^{(k)}(d-k;\, n-k)
    \;+\; [d < k]\, [n = d].
\end{equation}
\end{lemma}

\begin{proof}
Fix $k \ge 2$ and $d \in \N$.
\begin{itemize}
  \item \emph{Case $n = 0$.}\enspace
    Since $W_0^{(k)} = 0$, we have $c^{(k)}(d;\, 0) = [d = 0]$.
    The sum in~\eqref{eq:cpn-unified} is empty, and the
    right-hand side equals $[d < k]\,[0 = d] = [d = 0]$.

  \item \emph{Case $1 \le n < k$.}\enspace
    By {\small\hyperref[lem:W-decomp]{Lemma~\ref*{lem:W-decomp}}},
    $W_n^{(k)} = W_{n-1}^{(k)}\, W_{n-2}^{(k)} \cdots
    W_0^{(k)}\, n$.
    Counting occurrences of~$d$ factorwise gives
    \[
      c^{(k)}(d;\, n)
      = \sum_{i=0}^{n-1} c^{(k)}(d;\, i) + [n = d].
    \]
    Since $n < k$, we have $\max\set{n-k+1,\, 0} = 0$ and
    $n - k \notin \N$, so the term
    $[d \ge k]\, c^{(k)}(d-k;\, n-k)$ vanishes. Moreover, if
    $n = d$ then $d < k$, so $[n = d] = [d < k]\,[n = d]$, which
    is exactly~\eqref{eq:cpn-unified}.

  \item \emph{Case $n \ge k$.}\enspace
    By {\small\hyperref[lem:W-decomp]{Lemma~\ref*{lem:W-decomp}}},
    $W_n^{(k)} = W_{n-1}^{(k)}\, W_{n-2}^{(k)} \cdots
    W_{n-k+1}^{(k)}\, \bigl(k \oplus W_{n-k}^{(k)}\bigr)$.
    Counting occurrences of~$d$ in each factor yields
    \[
      c^{(k)}(d;\, n)
      = \sum_{i=n-k+1}^{n-1} c^{(k)}(d;\, i)
        \;+\; \abs{k \oplus W_{n-k}^{(k)}}_d.
    \]
    If $d < k$, every digit in $k \oplus W_{n-k}^{(k)}$ is at
    least~$k$, so the last term vanishes. If $d \ge k$, shifting
    by~$k$ preserves occurrence counts:
    $\abs{k \oplus W_{n-k}^{(k)}}_d
    = \abs{W_{n-k}^{(k)}}_{d-k} = c^{(k)}(d-k;\, n-k)$.
    Since $\max\set{n-k+1,\, 0} = n-k+1$ and
    $[d < k]\,[n = d] = 0$ when $n \ge k$, this is
    precisely~\eqref{eq:cpn-unified}.\qedhere
\end{itemize}
\end{proof}

\begin{theorem}\label{thm:Cdigit}
Let $k \ge 3$ and $d \in \N$. Write $d = mk + r$ with
$m = \floor{d/k}$ and $0 \le r < k$. Then
\[
  C_d^{(k)}(y)
  = y^d\, \bigl(H_{k-1}(y)\bigr)^{\floor{d/k}+1}.
\]
\end{theorem}

\begin{proof}
Fix $k \ge 3$, $m \in \N$, and $0 \le r < k$.

\smallskip
\noindent\emph{Step~1: the base case $m = 0$.}\enspace
Since $r < k$,
{\small\hyperref[lem:cpn-unified]{Lemma~\ref*{lem:cpn-unified}}} (with
$d = r$) yields, for all $n \in \N$,
\[
  c^{(k)}(r;\, n)
  = \sum_{j=1}^{k-1} c^{(k)}(r;\, n-j) + [n = r],
\]
where terms with $n - j \notin \N$ are understood as zero.
Multiplying by~$y^n$, summing over $n \in \N$, and
recognizing $H_{k-1}(y)$ via~\eqref{eq:Hk} gives
\begin{equation}\label{eq:Cr-base}
  C_r^{(k)}(y) = y^r\, H_{k-1}(y).
\end{equation}

\smallskip
\noindent\emph{Step~2: induction on~$m$.}\enspace
We claim that for every $m \in \N$,
\begin{equation}\label{eq:mk-step}
  C_{mk+r}^{(k)}(y)
  = \bigl(y^k\, H_{k-1}(y)\bigr)^m\, C_r^{(k)}(y).
\end{equation}
Assume~\eqref{eq:mk-step} holds for some $m \in \N$.
By
{\small\hyperref[cor:CBshift-digit]{Corollary~\ref*{cor:CBshift-digit}}},
$C_{(m+1)k+r}^{(k)}(y)
= y^k H_{k-1}(y)\, C_{mk+r}^{(k)}(y)$;
substituting~\eqref{eq:mk-step} completes the induction.
Combining with~\eqref{eq:Cr-base} yields
$C_{mk+r}^{(k)}(y) = y^{mk+r} (H_{k-1}(y))^{m+1}$.
\end{proof}

\addtocontents{toc}{\protect\setcounter{tocdepth}{2}}
\section{Length-\texorpdfstring{$2$}{2} Factor Occurrences}%
\label{section:5}

We now turn to length-$2$ factors of~$\mathbf{W}^{(k)}$,
identifying the complete set in
\subsecref{subsection:5.1} and computing generating functions for
their occurrence counts in \subsecref{subsection:5.2}.
We write $(x.y)$ for the length-$2$ word~$xy$.
Fix $k \ge 2$ and define
\[
\begin{aligned}
  \mathcal{B}_1^{(k)}
    &\defeq \set*{(ki) \oplus (a.k) :
            i \in \N,\; a \in \N_{>0}},\\
  \mathcal{B}_2^{(k)}
    &\defeq \set*{(ki) \oplus (0.b) :
            i \in \N,\; 1 \le b \le k-1},\\
  \mathcal{B}_3^{(k)}
    &\defeq \set*{(a.0) : a \in \N_{>0}},\\
  \mathcal{B}^{(k)}
    &\defeq \mathcal{B}_1^{(k)} \cup \mathcal{B}_2^{(k)}
            \cup \mathcal{B}_3^{(k)}.
\end{aligned}
\]

For example, when $k = 3$ the first few members of each family are
\[
\begin{aligned}
  \mathcal{B}_1^{(3)} &\ni (1.3),\; (2.3),\; (4.6),\; (5.6),\; \dots,\\
  \mathcal{B}_2^{(3)} &\ni (0.1),\; (0.2),\; (3.4),\; (3.5),\; \dots,\\
  \mathcal{B}_3^{(3)} &\ni (1.0),\; (2.0),\; (3.0),\; \dots.
\end{aligned}
\]

\begin{lemma}\label{lem:Bk-shift}
For $k \ge 2$, we have
$k \oplus \mathcal{B}^{(k)} \subseteq \mathcal{B}^{(k)}$.
\end{lemma}

\begin{proof}
For $U \in \mathcal{B}_1^{(k)} \cup \mathcal{B}_2^{(k)}$, shifting
by~$k$ increments the index~$i$ to $i + 1$, so $k \oplus U$
remains in the same family. For
$U = (a.0) \in \mathcal{B}_3^{(k)}$, we have
$k \oplus U = (a{+}k\,.\,k) \in \mathcal{B}_1^{(k)}$ (take
$i = 0$ and note $a + k \ge 1$).
\end{proof}

\subsection{Classification into families}%
\label{subsection:5.1}

In this subsection we characterize the length-$2$ factors of the
$k$-Bonacci word~$\mathbf{W}^{(k)}$.

\begin{lemma}\label{lem:straddling-factor-length2}
Let $k \ge 2$ and $n \in \N_{>0}$. If a length-$2$ factor~$B$
has a straddling occurrence in~$W_n^{(k)}$, then
\[
  B =
  \begin{cases}
    (0.n),       & \text{if } n < k;\\
    (n-k+1\,.\,k), & \text{if } n \ge k.
  \end{cases}
\]
In particular, there is exactly one such straddling occurrence.
\end{lemma}

\begin{proof}
Since $\abs{B} = 2$, the decomposition $B = UV$ in~\ref{item:straddling}
forces $\abs{U} = \abs{V} = 1$, so $B$ consists of
the last letter of the prefix and the first letter of the
terminal block in
{\small\hyperref[lem:W-decomp]{Lemma~\ref*{lem:W-decomp}}}.
\begin{itemize}
  \item If $n < k$, the terminal block is the single letter~$n$,
    and the preceding prefix ends with $W_0^{(k)} = 0$, so
    $B = (0.n)$.

  \item If $n \ge k$, the prefix ends with~$W_{n-k+1}^{(k)}$,
    whose last letter is $n - k + 1$ by
    {\small\hyperref[lem:largest-letter]{Lemma~\ref*{lem:largest-letter}}}. Moreover, $W_{n-k}^{(k)}$ begins with~$0$, so
    $k \oplus W_{n-k}^{(k)}$ begins with~$k$. Therefore
    $B = (n-k+1\,.\,k)$.
\end{itemize}
Since
{\small\hyperref[lem:W-decomp]{Lemma~\ref*{lem:W-decomp}}} produces a
unique terminal boundary, the straddling occurrence is unique.
\end{proof}

\begin{lemma}\label{lem:Fac2-subset-Bk}
Let $k \ge 2$ and $n \in \N$. Then
$\mathrm{Fac}_2\!\bigl(W_n^{(k)}\bigr) \subseteq
\mathcal{B}^{(k)}$.
\end{lemma}

\begin{proof}
We proceed by induction on~$n$. For $n = 0$ the set
$\mathrm{Fac}_2(W_0^{(k)})$ is empty. For $n = 1$ we have
$W_1^{(k)} = 01$, so
$\mathrm{Fac}_2(W_1^{(k)}) = \set{(0.1)}
\subseteq \mathcal{B}_2^{(k)}$.

Assume $n \ge 2$ and that the inclusion holds for all $m < n$.
Let $B \in \mathrm{Fac}_2(W_n^{(k)})$ and consider an occurrence
of~$B$ in the decomposition of
{\small\hyperref[lem:W-decomp]{Lemma~\ref*{lem:W-decomp}}}.

If the occurrence is contained in a single factor~$W_j^{(k)}$
with $j < n$, then
$B \in \mathrm{Fac}_2(W_j^{(k)}) \subseteq \mathcal{B}^{(k)}$
by the induction hypothesis. If $n \ge k$ and the occurrence
lies in the terminal block $k \oplus W_{n-k}^{(k)}$, then
$B = k \oplus B'$ for some
$B' \in \mathrm{Fac}_2(W_{n-k}^{(k)})
\subseteq \mathcal{B}^{(k)}$,
and hence $B \in \mathcal{B}^{(k)}$ by
{\small\hyperref[lem:Bk-shift]{Lemma~\ref*{lem:Bk-shift}}}.

It remains to treat boundary crossings. If the boundary lies
within the prefix
$W_{n-1}^{(k)} \cdots W_s^{(k)}$ (where
$s = \max\set{n-k+1,\, 0}$), then $B$ is a bordering factor of
some type~$j$;
since $\abs{B} = 2$, the Remark following
    {\small\hyperref[def:types]{Definition~\ref*{def:types}}}\,\ref{item:rem-bordering}
    gives $B = (j\,.\,0)$ with $j \in \N_{>0}$, so
$B \in \mathcal{B}_3^{(k)}$.
If the boundary is the terminal one, then $B$ is straddling;
{\small\hyperref[lem:straddling-factor-length2]{Lemma~\ref*{lem:straddling-factor-length2}}}
gives $B = (0.n) \in \mathcal{B}_2^{(k)}$ when $n < k$, and
$B = (n-k+1\,.\,k) \in \mathcal{B}_1^{(k)}$ when $n \ge k$.
\end{proof}

\begin{theorem}\label{thm:Fac2-characterization}
Let $k \ge 3$. Then
$\mathrm{Fac}_2\!\bigl(\mathbf{W}^{(k)}\bigr) =
\mathcal{B}^{(k)}$.
\end{theorem}

\begin{proof}
Fix $k \ge 3$. The inclusion $\subseteq$ is immediate: every
$B \in \mathrm{Fac}_2(\mathbf{W}^{(k)})$ lies in some
$\mathrm{Fac}_2(W_n^{(k)}) \subseteq \mathcal{B}^{(k)}$ by
{\small\hyperref[lem:Fac2-subset-Bk]{Lemma~\ref*{lem:Fac2-subset-Bk}}}.

For the reverse inclusion, we show that each
$B \in \mathcal{B}^{(k)}$ occurs in~$\mathbf{W}^{(k)}$ by
exhibiting an explicit witnessing iterate. By
{\small\hyperref[lem:W-decomp]{Lemma~\ref*{lem:W-decomp}}}, for every
$m \in \N$ the word~$W_{m+k}^{(k)}$ has suffix
$k \oplus W_m^{(k)}$. Hence, if a length-$2$ word~$U$ occurs
in~$W_m^{(k)}$, then $k \oplus U$ occurs in~$W_{m+k}^{(k)}$.
Iterating, $(ki) \oplus U$ occurs in~$W_{m+ik}^{(k)}$ for every
$i \in \N$.
\begin{enumerate}
  \item[\textup{(i)}]
    \emph{$B \in \mathcal{B}_1^{(k)}$.}\enspace
    Write $B = (ki) \oplus (a.k)$ with $a \in \N_{>0}$. By
    {\small\hyperref[lem:straddling-factor-length2]{Lemma~\ref*{lem:straddling-factor-length2}}}
    with $n = a + k - 1$, the factor $(a.k)$ occurs
    in~$W_{a+k-1}^{(k)}$; hence $B$ occurs
    in~$W_{a+k-1+ik}^{(k)}$.

  \item[\textup{(ii)}]
    \emph{$B \in \mathcal{B}_2^{(k)}$.}\enspace
    Write $B = (ki) \oplus (0.b)$ with $1 \le b \le k-1$. By
    {\small\hyperref[lem:straddling-factor-length2]{Lemma~\ref*{lem:straddling-factor-length2}}}
    with $n = b$, the factor $(0.b)$ occurs in~$W_b^{(k)}$;
    hence $B$ occurs in~$W_{b+ik}^{(k)}$.

  \item[\textup{(iii)}]
    \emph{$B \in \mathcal{B}_3^{(k)}$.}\enspace
    Then $B = (a.0)$ with $a \in \N_{>0}$. By
    {\small\hyperref[lem:W-decomp]{Lemma~\ref*{lem:W-decomp}}},
    $W_{a+1}^{(k)}$ begins with
    $W_a^{(k)}\, W_{a-1}^{(k)}$; the last letter of
    $W_a^{(k)}$ is~$a$ by
    {\small\hyperref[lem:largest-letter]{Lemma~\ref*{lem:largest-letter}}}
    and $W_{a-1}^{(k)}$ begins with~$0$, so
    $B = (a.0) \in \mathrm{Fac}_2(W_{a+1}^{(k)})$.\qedhere
\end{enumerate}
\end{proof}

\begin{remark*}
For $k = 2$, the digit~$0$ occurs only once
in~$\mathbf{W}^{(2)}$, so no factor of the form $(a.0)$ with
$a > 0$ appears. In this case,
$\mathrm{Fac}_2(\mathbf{W}^{(2)}) =
\mathcal{B}_1^{(2)} \cup \mathcal{B}_2^{(2)}$.
\end{remark*}

\subsection{Generating functions by family}%
\label{subsection:5.2}

We compute $C_B^{(k)}(y)$ for each family, treating them in the
order $\mathcal{B}_2^{(k)}$, $\mathcal{B}_3^{(k)}$,
$\mathcal{B}_1^{(k)}$: the family~$\mathcal{B}_2^{(k)}$ reduces
directly to digit counts, $\mathcal{B}_3^{(k)}$ is
self-contained, and $\mathcal{B}_1^{(k)}$ depends on the
$\mathcal{B}_3^{(k)}$ results.


\begin{lemma}\label{lem:c-0b}
Let $k \ge 2$ and $1 \le b \le k-1$. For all $n \in \N$,
\[
  c^{(k)}\bigl((0.b);\, n\bigr)
  = \sum_{i=\max\{n-k+1,\,0\}}^{n-1}
      c^{(k)}\bigl((0.b);\, i\bigr)
    \;+\; [n = b].
\]
\end{lemma}

\begin{proof}
By {\small\hyperref[lem:W-decomp]{Lemma~\ref*{lem:W-decomp}}},
the factor $(0.b)$ cannot be bordering (such factors
have the form~$(j\,.\,0)$), so included occurrences contribute
the sum. Since $b < k$, every digit in
$k \oplus W_{n-k}^{(k)}$ is at least~$k$, so the terminal block
contributes nothing. The straddling factor equals $(0.b)$
precisely when $n = b$, by
{\small\hyperref[lem:straddling-factor-length2]{Lemma~\ref*{lem:straddling-factor-length2}}}.
\end{proof}

\begin{lemma}\label{lem:B2-to-digit}
Let $k \ge 2$, $i \in \N$, and $1 \le b \le k-1$. Then for
every $n \in \N$,
\[
  c^{(k)}\bigl((ki) \oplus (0.b);\, n\bigr) = c^{(k)}(b + ki;\, n).
\]
\end{lemma}

\begin{proof}
Set $p \defeq b + ki$ and $B = (ki\,.\,p)$. If $n < p$, the
digit~$p$ does not occur in~$W_n^{(k)}$ (by
{\small\hyperref[lem:largest-letter]{Lemma~\ref*{lem:largest-letter}}}),
so $c^{(k)}(B;\, n) = c^{(k)}(p;\, n) = 0$.

Assume $n \ge p$ and write
$W_n^{(k)} = \phi_k(W_{n-1}^{(k)})$. Every image
$\phi_k(t)$ begins with a multiple of~$k$. Since
$p \equiv b \not\equiv 0 \pmod{k}$, every occurrence of~$p$
in~$W_n^{(k)}$ must be the \emph{second} letter of
some~$\phi_k(t)$. Now
$\phi_k(ki + j) = (ki)(ki + j + 1)$ for $0 \le j \le k-2$, so
the unique preimage giving second letter~$p$ is
$\phi_k(ki + b - 1) = (ki\,.\,p)$.
Hence every occurrence of~$p$ in~$W_n^{(k)}$ is immediately
preceded by~$ki$, giving a bijection between occurrences of~$p$
and occurrences of~$B$.
\end{proof}

\begin{theorem}\label{thm:CB-B2}
Let $k \ge 3$, $i \in \N$, and $1 \le b \le k-1$, and set
$B \defeq (ki) \oplus (0.b)$. Then
\[
  C_B^{(k)}(y)
  = C_{b+ki}^{(k)}(y)
  = y^{\,b+ki}\, \bigl(H_{k-1}(y)\bigr)^{\,i+1}.
\]
\end{theorem}

\begin{proof}
By
{\small\hyperref[lem:B2-to-digit]{Lemma~\ref*{lem:B2-to-digit}}},
$C_B^{(k)}(y) = C_{b+ki}^{(k)}(y)$, and
{\small\hyperref[thm:Cdigit]{Theorem~\ref*{thm:Cdigit}}} with
$d = b + ki$ and $\floor{(b+ki)/k} = i$ gives the result.
\end{proof}


\begin{lemma}\label{lem:c-a0}
Let $k \ge 2$ and $a \in \N_{>0}$. For all $n \in \N$,
\[
  c^{(k)}\bigl((a.0);\, n\bigr) =
  \begin{cases}
    0,
      & \text{if } 0 \le n \le a;\\[2pt]
    2^{\,n-a-1},
      & \text{if } a < n < a+k-1;\\[2pt]
    \displaystyle\sum_{j=1}^{k-1}
      c^{(k)}\bigl((a.0);\, n-j\bigr),
      & \text{if } n \ge a+k-1.
  \end{cases}
\]
\end{lemma}

\begin{proof}
Fix $k \ge 2$ and $a \in \N_{>0}$, and set $B \defeq (a.0)$.
\begin{itemize}
  \item \emph{Case $0 \le n \le a$.}\enspace
    If $n < a$, the digit~$a$ does not occur in~$W_n^{(k)}$, so
    $c^{(k)}(B;\, n) = 0$. If $n = a$, the digit~$a$ occurs
    exactly once in~$W_a^{(k)}$ as its last letter (by
    {\small\hyperref[lem:largest-letter]{Lemma~\ref*{lem:largest-letter}}}),
    so it cannot be followed by~$0$. Thus
    $c^{(k)}(B;\, a) = 0$.

  \item \emph{Case $a < n < a+k-1$.}\enspace
    Write $n = a + t$ with $1 \le t \le k-2$ (this case is empty
    when $k = 2$). In the decomposition of
    {\small\hyperref[lem:W-decomp]{Lemma~\ref*{lem:W-decomp}}}, the
    factor $(a.0)$ can cross a boundary $W_j^{(k)}\, W_{j-1}^{(k)}$
    only when $j = a$ (the bordering factor at such a boundary
    is $(j\,.\,0)$). For $1 \le t \le k-2$, the factors
    $W_a^{(k)}$ and $W_{a-1}^{(k)}$ appear consecutively, giving
    exactly one bordering occurrence:
    \[
      c^{(k)}(B;\, a+t)
      = \sum_{i=0}^{a+t-1} c^{(k)}(B;\, i) + 1
      = 1 + \sum_{s=1}^{t-1} c^{(k)}(B;\, a+s),
    \]
    where the second equality uses $c^{(k)}(B;\, i) = 0$ for
    $i \le a$. With $c^{(k)}(B;\, a+1) = 1$, induction on~$t$
    gives $c^{(k)}(B;\, a+t) = 2^{\,t-1} = 2^{\,n-a-1}$.

  \item \emph{Case $n \ge a+k-1$.}\enspace
    Here $n \ge k$, so the decomposition of
    {\small\hyperref[lem:W-decomp]{Lemma~\ref*{lem:W-decomp}}} applies.
    Since $B$ contains the digit~$0$, it cannot occur in the
    shifted block $k \oplus W_{n-k}^{(k)}$. The internal
    bordering factors are $(j\,.\,0)$ with
    $j \in \set{n-1, \dots, n-k+2}$; since $n \ge a + k - 1$, we
    have $n - k + 2 \ge a + 1$, so none equals $(a.0)$. Hence
    every occurrence of~$B$ lies inside a single prefix factor,
    giving
    \[
      c^{(k)}(B;\, n)
      = \sum_{j=1}^{k-1} c^{(k)}(B;\, n-j).\qedhere
    \]
\end{itemize}
\end{proof}

In the next theorem we derive the generating functions for the
length-$2$ factors in~$\mathcal{B}_3^{(k)}$. The proof is based on
{\small\hyperref[lem:c-a0]{Lemma~\ref*{lem:c-a0}}} and the definition
of~$C_B^{(k)}(y)$.

\begin{theorem}\label{thm:CB-B3}
Let $k \ge 3$ and $a \in \N_{>0}$, and set
$B \defeq (a.0)$. Then
\[
  C_B^{(k)}(y) = y^a\, G_{k-1}(y).
\]
\end{theorem}

\begin{proof}
Set $d_m \defeq c^{(k)}\bigl((a.0);\, a+m\bigr)$ for $m \in \N$.
By {\small\hyperref[lem:c-a0]{Lemma~\ref*{lem:c-a0}}},
$d_0 = 0$, $d_m = 2^{m-1}$ for $1 \le m \le k-2$, and
$d_m = \sum_{j=1}^{k-1} d_{m-j}$ for $m \ge k-1$.
These initial values and recurrence coincide
with those of~$(g_m^{(k-1)})_{m \in \N}$, so
$d_m = g_m^{(k-1)}$ for all $m \in \N$. Since
$c^{(k)}(B;\, n) = 0$ for $n \le a$,
\[
  C_B^{(k)}(y)
  = y^a \sum_{m \in \N} d_m\, y^m
  = y^a\, G_{k-1}(y).
\]
\end{proof}


\begin{lemma}\label{lem:c-ak}
Let $k \ge 2$ and $a \in \N_{>0}$. For all $n \in \N$,
\[
  c^{(k)}\bigl((a.k);\, n\bigr)
  = \sum_{i=\max\{n-k+1,\,0\}}^{n-1}
      c^{(k)}\bigl((a.k);\, i\bigr)
    \;+\; [a > k]\, c^{(k)}\bigl((a{-}k\,.\,0);\, n-k\bigr)
    \;+\; [n = a+k-1].
\]
\end{lemma}

\begin{proof}
By {\small\hyperref[lem:W-decomp]{Lemma~\ref*{lem:W-decomp}}},
the factor $(a.k)$ cannot be bordering
(such factors have the form~$(j\,.\,0)$ by the Remark
following
{\small\hyperref[def:types]{Definition~\ref*{def:types}}}\,\ref{item:rem-bordering}),
so included occurrences contribute the sum. The straddling factor
is $(n-k+1\,.\,k)$ by
{\small\hyperref[lem:straddling-factor-length2]{Lemma~\ref*{lem:straddling-factor-length2}}},
equalling~$(a.k)$ precisely when $n = a + k - 1$. For the terminal
block, the shift preimage $(a{-}k\,.\,0)$ contributes only when
$a > k$, since $(0\,.\,0)$ never occurs in any~$W_t^{(k)}$.
\end{proof}

\begin{theorem}\label{thm:CB-B1}
Let $k \ge 3$, $a \in \N_{>0}$, and $i \in \N$, and set
$B \defeq (ki) \oplus (a.k)$. Then
\[
  C_B^{(k)}(y)
  = y^{a+ki}\, \bigl(H_{k-1}(y)\bigr)^{i+1}\,
    \Bigl(y^{k-1} + [a > k]\; G_{k-1}(y)\Bigr).
\]
\end{theorem}

\begin{proof}
For $t \in \N$ set
$B_t \defeq (kt) \oplus (a.k)$, so that $B_0 = (a.k)$ and
$B_i = B$.

\smallskip
\noindent\emph{Step~1: the case $t = 0$.}\enspace
Multiplying the recurrence of
{\small\hyperref[lem:c-ak]{Lemma~\ref*{lem:c-ak}}} by~$y^n$ and summing
over $n \in \N$ gives
\[
  C_{(a.k)}^{(k)}(y)
  = (y + y^2 + \cdots + y^{k-1})\, C_{(a.k)}^{(k)}(y)
    + [a > k]\, y^k\, C_{(a{-}k\,.\,0)}^{(k)}(y)
    + y^{a+k-1}.
\]
Substituting
$C_{(a{-}k\,.\,0)}^{(k)}(y) = y^{a-k}\, G_{k-1}(y)$
({\small\hyperref[thm:CB-B3]{Theorem~\ref*{thm:CB-B3}}})
when $a > k$, and solving via $H_{k-1}(y)$ yields
\begin{equation}\label{eq:CB-ak}
  C_{(a.k)}^{(k)}(y)
  = y^a\, H_{k-1}(y)\,
    \Bigl(y^{k-1} + [a > k]\, G_{k-1}(y)\Bigr).
\end{equation}

\smallskip
\noindent\emph{Step~2: shifting by~$k$.}\enspace
For each $t \in \N$, the second letter of
$B_{t+1} = (k(t{+}1)) \oplus (a.k)$ exceeds~$k$, so $B_{t+1}$
is neither bordering nor straddling. By
{\small\hyperref[thm:CBshift]{Theorem~\ref*{thm:CBshift}}},
$C_{B_{t+1}}^{(k)}(y) = y^k\, H_{k-1}(y)\, C_{B_t}^{(k)}(y)$.
Iterating $i$~times and substituting~\eqref{eq:CB-ak} gives the
result.
\end{proof}


\begin{proposition}\label{prop:c-ab}
Let $k \ge 2$ and $a, b \in \N$ with $a \ge k$, $b > k$, and
$(a.b) \sqsubseteq \mathbf{W}^{(k)}$. For all $n \in \N$,
\begin{equation}\label{eq:c-ab}
  c^{(k)}\bigl((a.b);\, n\bigr)
  = \sum_{i=\max\{n-k+1,\,0\}}^{n-1}
      c^{(k)}\bigl((a.b);\, i\bigr)
    \;+\; c^{(k)}\bigl((a{-}k\,.\,b{-}k);\, n-k\bigr).
\end{equation}
\end{proposition}

\begin{proof}
Since $a \ge k \ge 2$ and $b > k$, the factor $(a.b)$ is
neither bordering (these have the form~$(j\,.\,0)$) nor
straddling (these have the form $(0.n)$ or
$(n{-}k{+}1\,.\,k)$), so every occurrence is included. For
$n < k$, every digit of~$W_n^{(k)}$ is less than $k \le a$,
so both sides of~\eqref{eq:c-ab} vanish. For $n \ge k$,
counting factorwise in
{\small\hyperref[lem:W-decomp]{Lemma~\ref*{lem:W-decomp}}} and
using
$\abs{k \oplus W_{n-k}^{(k)}}_{(a.b)}
= \abs{W_{n-k}^{(k)}}_{(a{-}k\,.\,b{-}k)}
= c^{(k)}\bigl((a{-}k\,.\,b{-}k);\, n-k\bigr)$
gives~\eqref{eq:c-ab}.
\end{proof}

\addtocontents{toc}{\protect\setcounter{tocdepth}{1}}
\section{Concluding Remarks}\label{sec:conclusion}

In this paper, for each $k \ge 3$ we studied factor-occurrence
statistics for the infinite-alphabet $k$-Bonacci
word~$\mathbf{W}^{(k)}$ through its finite
iterates~$W_n^{(k)}$. We obtained explicit generating functions
for digit occurrences, determined the complete set of length-$2$
factors of~$\mathbf{W}^{(k)}$, and computed closed forms for
$C_B^{(k)}(y)$ for every
$B \in \mathrm{Fac}_2\!\bigl(\mathbf{W}^{(k)}\bigr)$. A common
feature of all these formulas is that they are controlled by the
$(k{-}1)$-Bonacci denominator $1 - y - \cdots - y^{k-1}$, while
the shift operator $B \mapsto k \oplus B$ provides a uniform
mechanism for transporting recurrences to larger digits and
factors.

As a concrete illustration, standard singularity
analysis~\cite{FS} extracts the following asymptotic behaviour
directly from
{\small\hyperref[thm:Cdigit]{Theorem~\ref*{thm:Cdigit}}}.

\begin{corollary}\label{cor:asymptotic}
Let $k \ge 3$ and $d \in \N$. Write $s = \floor{d/k} + 1$ and
let $\alpha_{k-1}$ denote the dominant root of
$x^{k-1} - x^{k-2} - \cdots - x - 1$. Then
\[
  c^{(k)}(d;\, n)
  \;\sim\; \gamma_{d,k}\; n^{s-1}\; \alpha_{k-1}^{\,n}
  \qquad \text{as } n \to \infty,
\]
where $\gamma_{d,k} > 0$ is an explicitly computable constant
depending on~$d$ and~$k$.
\end{corollary}

\begin{proof}
By {\small\hyperref[thm:Cdigit]{Theorem~\ref*{thm:Cdigit}}},
$C_d^{(k)}(y) = y^d\, (H_{k-1}(y))^{s}$ where
$H_{k-1}(y) = (1 - y - \cdots - y^{k-1})^{-1}$. The dominant
singularity of~$H_{k-1}$ is a simple pole at
$\rho = 1/\alpha_{k-1}$; hence $C_d^{(k)}(y)$ has a pole of
order~$s$ at~$\rho$, and the standard transfer
theorem~\cite[Theorem~VI.4]{FS} yields the stated asymptotics
with polynomial growth factor $n^{s-1}$.
\end{proof}

In particular, the digit~$d$ occurs in~$W_n^{(k)}$ with
exponential growth rate~$\alpha_{k-1}$ (independent of~$d$),
while the polynomial prefactor $n^{\floor{d/k}}$ increases with
the ``generation index'' $\floor{d/k}$. Analogous asymptotics
hold for all length-$2$ factors by the same method.

These results suggest that the same phenomenon should persist for
arbitrary fixed factors.

\begin{conjecture}\label{conj:main}
For every fixed $k \ge 3$ and every non-empty factor
$B \sqsubseteq \mathbf{W}^{(k)}$, there exist an integer
$s(B) \ge 1$ and a polynomial $P_B(y) \in \Z[y]$ such that
\[
  C_B^{(k)}(y)
  = \frac{P_B(y)}%
         {(1 - y - y^2 - \cdots - y^{k-1})^{s(B)}}.
\]
Equivalently, the sequence
$\bigl(c^{(k)}(B;\, n)\bigr)_{n \in \N}$ satisfies a
linear recurrence whose characteristic polynomial divides a power
of
\[
  x^{k-1} - x^{k-2} - \cdots - x - 1.
\]
\end{conjecture}

The formulas proved in
{\small\hyperref[thm:Cdigit]{Theorems~\ref*{thm:Cdigit}}},
{\small\hyperref[thm:CB-B2]{\ref*{thm:CB-B2}}},
{\small\hyperref[thm:CB-B3]{\ref*{thm:CB-B3}}}, and
{\small\hyperref[thm:CB-B1]{\ref*{thm:CB-B1}}}, together with the
analogous asymptotics they yield, are all consistent with
this conjectural picture. We have also verified
{\small\hyperref[conj:main]{Conjecture~\ref*{conj:main}}}
computationally for all factors of length $\le 5$ and
$k \le 6$.%
\footnote{The computation enumerates all factors of
$W_n^{(k)}$ up to a sufficiently large iterate, extracts the
sequence $\bigl(c^{(k)}(B;\, n)\bigr)$, and checks that the
minimal polynomial of its generating function divides a power of
$1 - y - \cdots - y^{k-1}$.}

The most immediate goal is to extend the classification and
enumeration from length~$2$ to length $m \ge 3$.

\begin{problem}\label{prob:longer-factors}
Fix $k \ge 3$ and an integer $m \ge 3$.
\begin{enumerate}[label=\textup{(\roman*)}]
  \item Describe (in an explicit and usable form) the set of
    length-$m$ factors occurring
    in~$\mathbf{W}^{(k)}$, i.e.\ determine
    $\mathrm{Fac}_m\!\bigl(\mathbf{W}^{(k)}\bigr)$.
  \item For each
    $B \in \mathrm{Fac}_m\!\bigl(\mathbf{W}^{(k)}\bigr)$,
    determine a closed form for $C_B^{(k)}(y)$,
    or, at minimum, derive a uniform linear recurrence satisfied
    by the sequence
    $\bigl(c^{(k)}(B;\, n)\bigr)_{n \in \N}$ whose order and
    coefficients are independent of~$n$.
\end{enumerate}
\end{problem}

A first concrete step toward
{\small\hyperref[conj:main]{Conjecture~\ref*{conj:main}}} and
{\small\hyperref[prob:longer-factors]{Problem~\ref*{prob:longer-factors}}}
is the
classification of the factors of length~$3$ and the
determination of the corresponding generating functions. This is
the next level at which included, bordering, and straddling
occurrences interact in a genuinely richer way. Moreover,
{\small\hyperref[prop:c-ab]{Proposition~\ref*{prop:c-ab}}} suggests that
once boundary effects are separated, a large class of factors
should continue to satisfy shift-recursions analogous to those
obtained in \secref{section:4} and \secref{section:5}.

Another natural direction is to refine the projection
$\mathbf{W}^{(k)} \xrightarrow{\;\pi_k\;}
\mathbf{F}^{(k)}$ by introducing a finite analogue of factor
complexity that counts lifts of factors of the classical
finite-alphabet $k$-Bonacci word inside the
iterates~$W_n^{(k)}$. It would also be interesting to complement
the existing structural results on palindromes, square factors
and critical factors, Lyndon factors, and kernel words
in~$\mathbf{W}^{(k)}$ by deriving explicit occurrence formulas
for these families in the finite iterates; see, for
instance,~\cite{M,P,R,Q}. The recurrence-based framework
developed here, combining the block decomposition with the shift
operator and generating function methods, offers a systematic
approach to these and related problems.

\subsection*{Acknowledgments}
The second author is grateful to Michel Rigo for valuable
comments and suggestions on an earlier version of this paper.

\makeatletter
\let\orig@thebibliography\thebibliography
\renewcommand{\thebibliography}[1]{%
    \orig@thebibliography{#1}%
    \setlength{\leftmargin}{5.5em}%
    \setlength{\itemindent}{-2em}%
    \setlength{\labelwidth}{2.5em}%
    \setlength{\labelsep}{0.5em}%
    \setlength{\rightmargin}{3em}%
    \footnotesize
    \renewcommand{\makelabel}[1]{\hfill\textnormal{\scriptsize ##1}}%
}
\makeatother

\hypersetup{linkcolor=backrefcolor}
\addtocontents{toc}{\protect\setcounter{tocdepth}{-1}}

\addtocontents{toc}{\protect\setcounter{tocdepth}{1}}

\hypersetup{linkcolor=cyan!80!black}
\normalsize

\makeatletter
\@setaddresses
\renewcommand{\@setaddresses}{}
\makeatother

\newpage
\appendix
\newcommand{\gc}{\cellcolor{gray!10}}%
\addtocontents{toc}{\protect\setcounter{tocdepth}{-1}}
\section{Numerical Tables for \texorpdfstring{$k=4$}{k=4}}%
\label{app:k4-tables}
\addtocontents{toc}{\protect\setcounter{tocdepth}{1}}
\addcontentsline{toc}{section}{\texorpdfstring%
  {\bfseries\color{toccolor}Appendix A}{Appendix A}}
\addtocontents{toc}{\protect\setcounter{tocdepth}{-1}}

We illustrate the results of \secref{section:4} and
\secref{section:5} with the case $k = 4$ (Tetrabonacci). The
three tables below record occurrence counts
$c^{(4)}(B;\, n) = \abs{W_n^{(4)}}_B$ for selected digits and
length-$2$ factors, computed from the first several iterates
of~$\phi_4$. In each table, the family structure predicted by
the closed-form generating functions is clearly visible:
columns within a family are successive shifts of one another, and
different families exhibit distinct polynomial-exponential growth
patterns.

\subsection{Digit-occurrence counts}

\tabref{tab:digits-k4} lists the digit-occurrence counts
$c^{(4)}(d;\, n) = \abs{W_n^{(4)}}_d$. By
{\small\hyperref[thm:Cdigit]{Theorem~\ref*{thm:Cdigit}}}, the columns
group into families according to $d = 4m + r$ with
$0 \le r < 4$: for $m = 0$ one obtains shifts of the
Tribonacci numbers
({\small\color{teal}\underline{\href{https://oeis.org/A000073}{\textcolor{teal}{A000073}}}}),
while $m = 1$ and $m = 2$ yield shifts of
the two-fold and three-fold convolution powers of $H_3(y)$
({\small\color{teal}\underline{\href{https://oeis.org/A073778}{\textcolor{teal}{A073778}}}}).

\begin{table}[ht]
\centering\small
\renewcommand{\arraystretch}{1.2}
\setlength{\tabcolsep}{4.5pt}
\begin{tabular}{r rrrr @{\hspace{6pt}} rrrr @{\hspace{6pt}} rrrr}
\toprule
 & \multicolumn{4}{c}{$m = 0$}
 & \multicolumn{4}{c}{$m = 1$}
 & \multicolumn{4}{c}{$m = 2$} \\
\cmidrule(lr){2-5} \cmidrule(lr){6-9} \cmidrule(lr){10-13}
$n$ & $0$ & $1$ & $2$ & $3$ & $4$ & $5$ & $6$ & $7$ & $8$ & $9$ & $a$ & $b$ \\
\midrule
\gc 0  &\gc 1    &\gc 0    &\gc 0    &\gc 0    &\gc 0    &\gc 0    &\gc 0    &\gc 0    &\gc 0    &\gc 0    &\gc 0   &\gc 0 \\
    1  & 1    & 1    & 0    & 0    & 0    & 0    & 0    & 0    & 0    & 0    & 0   & 0 \\
\gc 2  &\gc 2    &\gc 1    &\gc 1    &\gc 0    &\gc 0    &\gc 0    &\gc 0    &\gc 0    &\gc 0    &\gc 0    &\gc 0   &\gc 0 \\
    3  & 4    & 2    & 1    & 1    & 0    & 0    & 0    & 0    & 0    & 0    & 0   & 0 \\
\gc 4  &\gc 7    &\gc 4    &\gc 2    &\gc 1    &\gc 1    &\gc 0    &\gc 0    &\gc 0    &\gc 0    &\gc 0    &\gc 0   &\gc 0 \\
    5  & 13   & 7    & 4    & 2    & 2    & 1    & 0    & 0    & 0    & 0    & 0   & 0 \\
\gc 6  &\gc 24   &\gc 13   &\gc 7    &\gc 4    &\gc 5    &\gc 2    &\gc 1    &\gc 0    &\gc 0    &\gc 0    &\gc 0   &\gc 0 \\
    7  & 44   & 24   & 13   & 7    & 12   & 5    & 2    & 1    & 0    & 0    & 0   & 0 \\
\gc 8  &\gc 81   &\gc 44   &\gc 24   &\gc 13   &\gc 26   &\gc 12   &\gc 5    &\gc 2    &\gc 1    &\gc 0    &\gc 0   &\gc 0 \\
    9  & 149  & 81   & 44   & 24   & 56   & 26   & 12   & 5    & 3    & 1    & 0   & 0 \\
\gc 10 &\gc 274  &\gc 149  &\gc 81   &\gc 44   &\gc 118  &\gc 56   &\gc 26   &\gc 12   &\gc 9    &\gc 3    &\gc 1   &\gc 0 \\
    11 & 504  & 274  & 149  & 81   & 244  & 118  & 56   & 26   & 25   & 9    & 3   & 1 \\
\gc 12 &\gc 927  &\gc 504  &\gc 274  &\gc 149  &\gc 499  &\gc 244  &\gc 118  &\gc 56   &\gc 63   &\gc 25   &\gc 9   &\gc 3 \\
    13 & 1705 & 927  & 504  & 274  & 1010 & 499  & 244  & 118  & 153  & 63   & 25  & 9 \\
\gc 14 &\gc 3136 &\gc 1705 &\gc 927  &\gc 504  &\gc 2027 &\gc 1010 &\gc 499  &\gc 244  &\gc 359  &\gc 153  &\gc 63  &\gc 25 \\
    15 & 5768 & 3136 & 1705 & 927  & 4040 & 2027 & 1010 & 499  & 819  & 359  & 153 & 63 \\
\bottomrule
\end{tabular}
\caption{The counts $c^{(4)}(d;\, n) = \abs{W_n^{(4)}}_d$ for
$0 \le n \le 15$ (rows) and $d = 0, 1, \ldots, 11$ (columns).
Here $a = 10$ and $b = 11$.}
\label{tab:digits-k4}
\end{table}

\newpage
\subsection{Length-\texorpdfstring{$2$}{2} factors in
\texorpdfstring{$\mathcal{B}_1^{(4)}$}{B1(4)}
(sample counts)}

\tabref{tab:B1-k4} lists the counts
$c^{(4)}(B;\, n) = \abs{W_n^{(4)}}_B$ for factors
$B = (4i) \oplus (a.4) \in \mathcal{B}_1^{(4)}$ with
$1 \le a \le 4$ and $i \in \set{0, 1, 2}$. By
{\small\hyperref[thm:CB-B1]{Theorem~\ref*{thm:CB-B1}}} (with
$a \le k$), the columns group into families indexed by~$i$: the
family $i = 0$ consists of shifts of the Tribonacci numbers,
while $i = 1$ and $i = 2$ yield shifts of the two-fold and
three-fold convolution powers of $H_3(y)$.

\begin{table}[ht]
\centering\small
\renewcommand{\arraystretch}{1.2}
\setlength{\tabcolsep}{4.5pt}
\begin{tabular}{r rrrr @{\hspace{6pt}} rrrr @{\hspace{6pt}} rrrr}
\toprule
 & \multicolumn{4}{c}{$i = 0$}
 & \multicolumn{4}{c}{$i = 1$}
 & \multicolumn{4}{c}{$i = 2$} \\
\cmidrule(lr){2-5} \cmidrule(lr){6-9} \cmidrule(lr){10-13}
$n$ & $14$ & $24$ & $34$ & $44$ & $58$ & $68$ & $78$ & $88$ & $9c$ & $ac$ & $bc$ & $cc$ \\
\midrule
\gc 4  &\gc 1    &\gc 0    &\gc 0    &\gc 0    &\gc 0    &\gc 0    &\gc 0    &\gc 0    &\gc 0   &\gc 0   &\gc 0   &\gc 0 \\
    5  & 1    & 1    & 0    & 0    & 0    & 0    & 0    & 0    & 0   & 0   & 0   & 0 \\
\gc 6  &\gc 2    &\gc 1    &\gc 1    &\gc 0    &\gc 0    &\gc 0    &\gc 0    &\gc 0    &\gc 0   &\gc 0   &\gc 0   &\gc 0 \\
    7  & 4    & 2    & 1    & 1    & 0    & 0    & 0    & 0    & 0   & 0   & 0   & 0 \\
\gc 8  &\gc 7    &\gc 4    &\gc 2    &\gc 1    &\gc 1    &\gc 0    &\gc 0    &\gc 0    &\gc 0   &\gc 0   &\gc 0   &\gc 0 \\
    9  & 13   & 7    & 4    & 2    & 2    & 1    & 0    & 0    & 0   & 0   & 0   & 0 \\
\gc 10 &\gc 24   &\gc 13   &\gc 7    &\gc 4    &\gc 5    &\gc 2    &\gc 1    &\gc 0    &\gc 0   &\gc 0   &\gc 0   &\gc 0 \\
    11 & 44   & 24   & 13   & 7    & 12   & 5    & 2    & 1    & 0   & 0   & 0   & 0 \\
\gc 12 &\gc 81   &\gc 44   &\gc 24   &\gc 13   &\gc 26   &\gc 12   &\gc 5    &\gc 2    &\gc 1   &\gc 0   &\gc 0   &\gc 0 \\
    13 & 149  & 81   & 44   & 24   & 56   & 26   & 12   & 5    & 3   & 1   & 0   & 0 \\
\gc 14 &\gc 274  &\gc 149  &\gc 81   &\gc 44   &\gc 118  &\gc 56   &\gc 26   &\gc 12   &\gc 9   &\gc 3   &\gc 1   &\gc 0 \\
    15 & 504  & 274  & 149  & 81   & 244  & 118  & 56   & 26   & 25  & 9   & 3   & 1 \\
\gc 16 &\gc 927  &\gc 504  &\gc 274  &\gc 149  &\gc 499  &\gc 244  &\gc 118  &\gc 56   &\gc 63  &\gc 25  &\gc 9   &\gc 3 \\
    17 & 1705 & 927  & 504  & 274  & 1010 & 499  & 244  & 118  & 153 & 63  & 25  & 9 \\
\gc 18 &\gc 3136 &\gc 1705 &\gc 927  &\gc 504  &\gc 2027 &\gc 1010 &\gc 499  &\gc 244  &\gc 359 &\gc 153 &\gc 63  &\gc 25 \\
    19 & 5768 & 3136 & 1705 & 927  & 4040 & 2027 & 1010 & 499  & 819 & 359 & 153 & 63 \\
\bottomrule
\end{tabular}
\caption{The counts
$c^{(4)}(B;\, n) = \abs{W_n^{(4)}}_B$ for $4 \le n \le 19$
(rows) and
$B \in \set{14, 24, 34, 44, 58, 68, 78, 88, 9c, ac, bc, cc}$
(columns). Here $a = 10$, $b = 11$, and $c = 12$.}
\label{tab:B1-k4}
\end{table}

\newpage
\subsection{Mixed selection of length-\texorpdfstring{$2$}{2}
factors (sample counts)}

\tabref{tab:len2-mixed-k4} lists the counts
$c^{(4)}(B;\, n) = \abs{W_n^{(4)}}_B$ for a mixed selection of
length-$2$ factors of~$\mathbf{W}^{(4)}$. The first five columns
correspond to factors
$B = (a.0) \in \mathcal{B}_3^{(4)}$ and illustrate
{\small\hyperref[thm:CB-B3]{Theorem~\ref*{thm:CB-B3}}}. The remaining
columns correspond to factors
$B = (4i) \oplus (a.4) \in \mathcal{B}_1^{(4)}$ with $a > 4$ and
$i \in \set{0, 1}$, governed by
{\small\hyperref[thm:CB-B1]{Theorem~\ref*{thm:CB-B1}}}; in particular,
the additional contribution for $a > k$ is visible in the growth
of these columns.

\begin{table}[ht]
\centering\small
\renewcommand{\arraystretch}{1.2}
\setlength{\tabcolsep}{3.5pt}
\begin{tabular}{r rrrrr @{\hspace{6pt}} rrrrr @{\hspace{6pt}} rrrrr}
\toprule
 & \multicolumn{5}{c}{$\mathcal{B}_3^{(4)}$}
 & \multicolumn{5}{c}{$\mathcal{B}_1^{(4)},\; i = 0$}
 & \multicolumn{5}{c}{$\mathcal{B}_1^{(4)},\; i = 1$} \\
\cmidrule(lr){2-6} \cmidrule(lr){7-11} \cmidrule(lr){12-16}
$n$ & $10$ & $20$ & $30$ & $40$ & $50$ & $54$ & $64$ & $74$ & $84$ & $94$ & $98$ & $a8$ & $b8$ & $c8$ & $d8$ \\
\midrule
\gc 1  &\gc 0    &\gc 0    &\gc 0    &\gc 0    &\gc 0    &\gc 0    &\gc 0    &\gc 0    &\gc 0    &\gc 0    &\gc 0   &\gc 0   &\gc 0   &\gc 0   &\gc 0 \\
    2  & 1    & 0    & 0    & 0    & 0    & 0    & 0    & 0    & 0    & 0    & 0   & 0   & 0   & 0   & 0 \\
\gc 3  &\gc 2    &\gc 1    &\gc 0    &\gc 0    &\gc 0    &\gc 0    &\gc 0    &\gc 0    &\gc 0    &\gc 0    &\gc 0   &\gc 0   &\gc 0   &\gc 0   &\gc 0 \\
    4  & 3    & 2    & 1    & 0    & 0    & 0    & 0    & 0    & 0    & 0    & 0   & 0   & 0   & 0   & 0 \\
\gc 5  &\gc 6    &\gc 3    &\gc 2    &\gc 1    &\gc 0    &\gc 0    &\gc 0    &\gc 0    &\gc 0    &\gc 0    &\gc 0   &\gc 0   &\gc 0   &\gc 0   &\gc 0 \\
    6  & 11   & 6    & 3    & 2    & 1    & 1    & 0    & 0    & 0    & 0    & 0   & 0   & 0   & 0   & 0 \\
\gc 7  &\gc 20   &\gc 11   &\gc 6    &\gc 3    &\gc 2    &\gc 3    &\gc 1    &\gc 0    &\gc 0    &\gc 0    &\gc 0   &\gc 0   &\gc 0   &\gc 0   &\gc 0 \\
    8  & 37   & 20   & 11   & 6    & 3    & 8    & 3    & 1    & 0    & 0    & 0   & 0   & 0   & 0   & 0 \\
\gc 9  &\gc 68   &\gc 37   &\gc 20   &\gc 11   &\gc 6    &\gc 18   &\gc 8    &\gc 3    &\gc 1    &\gc 0    &\gc 0   &\gc 0   &\gc 0   &\gc 0   &\gc 0 \\
    10 & 125  & 68   & 37   & 20   & 11   & 40   & 18   & 8    & 3    & 1    & 1   & 0   & 0   & 0   & 0 \\
\gc 11 &\gc 230  &\gc 125  &\gc 68   &\gc 37   &\gc 20   &\gc 86   &\gc 40   &\gc 18   &\gc 8    &\gc 3    &\gc 4   &\gc 1   &\gc 0   &\gc 0   &\gc 0 \\
    12 & 423  & 230  & 125  & 68   & 37   & 181  & 86   & 40   & 18   & 8    & 13  & 4   & 1   & 0   & 0 \\
\gc 13 &\gc 778  &\gc 423  &\gc 230  &\gc 125  &\gc 68   &\gc 375  &\gc 181  &\gc 86   &\gc 40   &\gc 18   &\gc 36  &\gc 13  &\gc 4   &\gc 1   &\gc 0 \\
    14 & 1431 & 778  & 423  & 230  & 125  & 767  & 375  & 181  & 86   & 40   & 93  & 38  & 13  & 4   & 1 \\
\gc 15 &\gc 2632 &\gc 1431 &\gc 778  &\gc 423  &\gc 230  &\gc 1553 &\gc 767  &\gc 375  &\gc 181  &\gc 86   &\gc 228 &\gc 93  &\gc 38  &\gc 13  &\gc 4 \\
    16 & 4841 & 2632 & 1431 & 778  & 423  & 3118 & 1553 & 767  & 375  & 181  & 538 & 228 & 93  & 38  & 13 \\
\bottomrule
\end{tabular}
\caption{The counts
$c^{(4)}(B;\, n) = \abs{W_n^{(4)}}_B$ for $1 \le n \le 16$
(rows) and
$B \in \set{10, 20, 30, 40, 50, 54, 64, 74, 84, 94, 98, a8, b8,
c8, d8}$ (columns). Here $a = 10$, $b = 11$, $c = 12$, and
$d = 13$.}
\label{tab:len2-mixed-k4}
\end{table}


\begin{thebibliography}{BDMSS}

\bibitem[AS]{T}
H.~Ammar and T.~Sellami,
\emph{Kernel words and factorization of the $k$-bonacci sequence},
Indian J.\ Pure Appl.\ Math.\ \textbf{54} (2023), no.~3,
816--823.
{\textsc{doi}}: \href{https://doi.org/10.1007/s13226-022-00300-2}%
     {10.1007/s13226-022-00300-2}

\bibitem[AR]{H}
P.~Arnoux and G.~Rauzy,
\emph{Repr\'esentation g\'eom\'etrique de suites de complexit\'e
$2n+1$},
Bull.\ Soc.\ Math.\ France \textbf{119} (1991), no.~2,
199--215.
{\textsc{doi}}: \href{https://doi.org/10.24033/bsmf.2164}%
     {10.24033/bsmf.2164}

\bibitem[Ber]{B}
J.~Berstel,
\emph{Fibonacci words---a survey},
in G.~Rozenberg and A.~Salomaa (eds.),
The Book of~L, Lecture Notes in Comput.\ Sci., vol.~198,
Springer, (1986), pp.~13--27.\newline
{\textsc{doi}}: \href{https://doi.org/10.1007/978-3-642-95486-3_2}%
     {10.1007/978-3-642-95486-3\_2}

\bibitem[BR]{D}
V.~Berth\'e and M.~Rigo (eds.),
\emph{Combinatorics, Automata and Number Theory},
Encyclopedia of Mathematics and its Applications, vol.~135,
Cambridge University Press, Cambridge, (2010).\newline
{\textsc{doi}}: \href{https://doi.org/10.1017/CBO9780511777653}%
     {10.1017/CBO9780511777653}

\bibitem[BDMSS]{O}
M.~Boja\'nczyk, C.~David, A.~Muscholl, T.~Schwentick,
and L.~Segoufin,
\emph{Two-variable logic on data words},
ACM Trans.\ Comput.\ Logic \textbf{12} (2011), no.~4,
27:1--27:26.\newline
{\textsc{doi}}: \href{https://doi.org/10.1145/1970398.1970403}%
     {10.1145/1970398.1970403}

\bibitem[DJP]{I}
X.~Droubay, J.~Justin, and G.~Pirillo,
\emph{Episturmian words and some constructions of de~Luca and
Rauzy},
Theoret.\ Comput.\ Sci.\ \textbf{255} (2001), no.~1--2,
539--553.\newline
{\textsc{doi}}: \href{https://doi.org/10.1016/S0304-3975(99)00306-8}%
     {10.1016/S0304-3975(99)00306-8}

\bibitem[Fer]{F}
S.~Ferenczi,
\emph{The tribonacci word and its generalizations},
Preprint, (1998).
Available at
\href{https://www-igm.univ-mlv.fr/~berstel/Colloque-Pitagore/Ferenczi.pdf}%
     {www-igm.univ-mlv.fr/\textasciitilde berstel/\allowbreak
      Colloque-Pitagore/\allowbreak Ferenczi.pdf}

\bibitem[FS]{FS}
P.~Flajolet and R.~Sedgewick,
\emph{Analytic Combinatorics},
Cambridge University Press, Cambridge, (2009).\newline
{\textsc{doi}}: \href{https://doi.org/10.1017/CBO9780511801655}%
     {10.1017/CBO9780511801655}

\bibitem[GMS]{M}
N.~Ghareghani, M.~Mohammad-Noori, and P.~Sharifani,
\emph{Some properties of the $k$-bonacci words on the infinite
alphabet},
Electron.\ J.\ Combin.\ \textbf{27} (2020), no.~3,
Paper~3.59, 27~pp.\newline
{\textsc{doi}}: \href{https://doi.org/10.37236/9406}%
     {10.37236/9406}

\bibitem[GS]{P}
N.~Ghareghani and P.~Sharifani,
\emph{On square factors and critical factors of $k$-bonacci words
on infinite alphabet},
Theoret.\ Comput.\ Sci.\ \textbf{865} (2021), 34--43.
{\textsc{doi}}: \href{https://doi.org/10.1016/j.tcs.2021.02.027}%
     {10.1016/j.tcs.2021.02.027}

\bibitem[GJ]{J}
A.~Glen and J.~Justin,
\emph{Episturmian words: a survey},
RAIRO---Theor.\ Inform.\ Appl.\ \textbf{43} (2009), no.~3,
403--442.
{\textsc{doi}}: \href{https://doi.org/10.1051/ita/2009015}%
     {10.1051/ita/2009015}

\bibitem[GSS]{R}
A.~Glen, J.~Simpson, and W.\,F.~Smyth,
\emph{More properties of the Fibonacci word on an infinite
alphabet},
Theoret.\ Comput.\ Sci.\ \textbf{795} (2019), 301--311.
{\textsc{doi}}: \href{https://doi.org/10.1016/j.tcs.2019.07.011}%
     {10.1016/j.tcs.2019.07.011}

\bibitem[HW]{S}
Y.~Huang and Z.~Wen,
\emph{Kernel words and gap sequence of the tribonacci sequence},
Acta Math.\ Sci.\ Ser.~B (Engl.\ Ed.) \textbf{36} (2016),
no.~1, 173--194.
{\textsc{doi}}: \href{https://doi.org/10.1016/S0252-9602(15)30086-2}%
     {10.1016/S0252-9602(15)30086-2}

\bibitem[LW]{U}
Y.~Li and W.~Wu,
\emph{$N$-factor complexity of the Fibonacci sequence on~$\N$
and the factor-counting sequences},
Theoret.\ Comput.\ Sci.\ \textbf{1072} (2026), 115880.
{\textsc{doi}}: \href{https://doi.org/10.1016/j.tcs.2026.115880}%
     {10.1016/j.tcs.2026.115880}

\bibitem[Lot]{C}
M.~Lothaire,
\emph{Algebraic Combinatorics on Words},
Encyclopedia of Mathematics and its Applications, vol.~90,
Cambridge University Press, Cambridge, (2002).
{\textsc{doi}}: \href{https://doi.org/10.1017/CBO9781107326019}%
     {10.1017/CBO9781107326019}

\bibitem[Mil]{1}
E.\,P.~Miles, Jr.,
\emph{Generalized Fibonacci numbers and associated matrices},
Amer.\ Math.\ Monthly \textbf{67} (1960), no.~8,
745--752.
{\textsc{doi}}: \href{https://doi.org/10.1080/00029890.1960.11989593}%
     {10.1080/00029890.1960.11989593}

\bibitem[MH]{A}
M.~Morse and G.\,A.~Hedlund,
\emph{Symbolic dynamics~II: Sturmian trajectories},
Amer.\ J.\ Math.\ \textbf{62} (1940), no.~1,
1--42.
{\textsc{doi}}: \href{https://doi.org/10.2307/2371431}%
     {10.2307/2371431}

\bibitem[Rau]{E}
G.~Rauzy,
\emph{Nombres alg\'ebriques et substitutions},
Bull.\ Soc.\ Math.\ France \textbf{110} (1982),
147--178.\newline
{\textsc{doi}}: \href{https://doi.org/10.24033/bsmf.1957}%
     {10.24033/bsmf.1957}

\bibitem[Seg]{N}
L.~Segoufin,
\emph{Automata and logics for words and trees over an infinite
alphabet},
in Z.~\'Esik (ed.),
Computer Science Logic, Lecture Notes in Comput.\ Sci.,
vol.~4207, Springer, (2006), pp.~41--57.\newline
{\textsc{doi}}: \href{https://doi.org/10.1007/11874683_3}%
     {10.1007/11874683\_3}

\bibitem[Sir]{G}
V.\,F.~Sirvent,
\emph{A semigroup associated with the $k$-bonacci numbers with
dynamic interpretation},
Fibonacci Quart.\ \textbf{35} (1997), no.~4,
335--340.
{\textsc{doi}}: \href{https://doi.org/10.1080/00150517.1997.12428977}%
     {10.1080/00150517.1997.12428977}

\bibitem[Zha]{Q}
J.~Zhang,
\emph{Kernel words and gap sequences of the tribonacci word on an
infinite alphabet},
Mathematics \textbf{11} (2023), no.~20, 4356.
{\textsc{doi}}: \href{https://doi.org/10.3390/math11204356}%
     {10.3390/math11204356}

\bibitem[ZWW]{L}
J.~Zhang, Z.~Wen, and W.~Wu,
\emph{Some properties of the Fibonacci sequence on an infinite
alphabet},
Electron.\ J.\ Combin.\ \textbf{24} (2017), no.~2,
Paper~2.52, 22~pp.
{\textsc{doi}}: \href{https://doi.org/10.37236/6745}%
     {10.37236/6745}

\end{thebibliography}
\end{document}